\documentclass[12pt,a4paper]{article}%
\usepackage[utf8]{inputenc}
\usepackage{amsmath}
\usepackage{amsfonts}
\usepackage{amssymb}
\usepackage{graphicx}
\usepackage[space]{grffile}
\usepackage{hyperref}%
\providecommand{\U}[1]{\protect\rule{.1in}{.1in}}
\newtheorem{theorem}{Theorem}
\newtheorem{theorem*}{Example}

\newtheorem{conjecture}[theorem]{Conjecture}
\newtheorem{corollary}[theorem]{Corollary}

\newtheorem{definition}[theorem]{Definition}

\newtheorem{lemma}[theorem]{Lemma}

\newtheorem{proposition}[theorem]{Proposition}
\newtheorem{remark}[theorem]{Observation}

\newenvironment{proof}[1][Proof]{\noindent\textbf{#1.} }{\ \hfill \rule{0.5em}{0.5em}\bigskip}
\graphicspath{{D:/Dropbox/Riste-Sedlar/Totally Irregular/Paper3/Slike/}}
\begin{document}

\title{Local Irregularity Conjecture vs. cacti}
\author{Jelena Sedlar$^{1,3}$,\\Riste \v Skrekovski$^{2,3}$ \\[0.3cm] {\small $^{1}$ \textit{University of Split, Faculty of civil
engineering, architecture and geodesy, Croatia}}\\[0.1cm] {\small $^{2}$ \textit{University of Ljubljana, FMF, 1000 Ljubljana,
Slovenia }}\\[0.1cm] {\small $^{3}$ \textit{Faculty of Information Studies, 8000 Novo
Mesto, Slovenia }}\\[0.1cm] }
\maketitle

\begin{abstract}
A graph is locally irregular if the degrees of the end-vertices of every edge
are distinct. An edge coloring of a graph $G$ is locally irregular if every
color induces a locally irregular subgraph of $G$. A colorable graph $G$ is
any graph which admits a locally irregular edge coloring. The locally
irregular chromatic index $\chi_{%
%TCIMACRO{\TeXButton{TeX field}{\rm{irr}}}%
%BeginExpansion
\rm{irr}%
%EndExpansion
}^{\prime}(G)$ of a colorable graph $G$ is the smallest number of colors
required by a locally irregular edge coloring of $G$. The Local Irregularity
Conjecture claims that all colorable graphs require at most $3$ colors for
locally irregular edge coloring. Recently, it has been observed that the
conjecture does not hold for the bow-tie graph $B$ \cite{SedSkreMasna}. Cacti
are important class of graphs for this conjecture since $B$ and all
non-colorable graphs are cacti. In this paper we show that for every colorable
cactus graph $G\not =B$ it holds that $\chi_{%
%TCIMACRO{\TeXButton{TeX field}{\rm{irr}}}%
%BeginExpansion
\rm{irr}%
%EndExpansion
}^{\prime}(G)\leq3.$ This makes us to believe that $B$ is the only colorable
graph with $\chi_{%
%TCIMACRO{\TeXButton{TeX field}{\rm{irr}}}%
%BeginExpansion
\rm{irr}%
%EndExpansion
}^{\prime}(B)>3$, and consequently that $B$ is the only counterexample to the
Local Irregularity Conjecture.

\end{abstract}

\textit{Keywords:} locally irregular edge coloring; Local Irregularity
Conjecture; cactus graphs.

\textit{AMS Subject Classification numbers:} 05C15

\section{Introduction}

All graphs mentioned in this paper are considered to be simple, finite and
connected. The main interest of this paper are edge colorings of a graph. A
graph is said to be \emph{locally irregular} if the degrees of the two
end-vertices of every edge are distinct. A \emph{locally irregular }%
$k$\emph{-edge coloring}, or $k$\emph{-liec} for short, is any $k$-edge
coloring of $G$ every color of which induces a locally irregular subgraph of
$G$. Since in this paper we will deal only with the locally irregular edge
colorings, a graph which admits such a coloring will be called
\emph{colorable}. The \emph{locally irregular chromatic index} $\chi_{%
%TCIMACRO{\TeXButton{TeX field}{\rm{irr}}}%
%BeginExpansion
\rm{irr}%
%EndExpansion
}^{\prime}(G)$ of a colorable graph $G$ is defined as the smallest $k$ such
that $G$ admits a $k$-liec. Notice that isolated vertices of a graph $G$ do
not influence the local irregularity of an edge coloring, so in every graph we
will consider if such vertices arise by edge deletion we will ignore them.

The first question is to establish which graphs are colorable. To answer this
question, we first need to introduce a special class $\mathfrak{T}$ of cactus
graphs. First, a \emph{cactus graph} is any graph with edge disjoint cycles.
Now, a class $\mathfrak{T}$ is defined as follows:

\begin{itemize}
\item the triangle $K_{3}$ is contained in $\mathfrak{T}$,

\item for every graph $G\in\mathfrak{T}$, a graph $H$ which also belongs to
$\mathfrak{T}$ can be constructed in a following way: a vertex $v\in V(G)$ of
degree $2,$ which belongs to a triangle of $G,$ is identified with an
end-vertex of an even length path or with the only vertex of degree one in a
graph consisting of a triangle and an odd length path hanging at one of the
vertices of triangle.
\end{itemize}

\noindent Obviously, graphs from $\mathfrak{T}$ are cacti in which cycles are
vertex disjoint triangles, which are connected by an odd length paths, and
besides triangles and odd length paths connecting them a cactus graph $G$ from
$\mathfrak{T}$ may have only even length paths hanging at any vertex $v$ such
that $v$ belongs to a triangle of $G$ and $d_{G}(v)=3.$ It was established
that the only non-colorable graphs are odd length paths, odd length cycles and
cactus graphs from $\mathfrak{T}$ \cite{Baudon5}. That settles the question of
non-colorable graphs, and for colorable graphs the interesting problem is to
determine the smallest number of colors required by a locally irregular edge
coloring of any graph. Regarding this problem, the following conjecture was
proposed \cite{Baudon5}.

\begin{conjecture}
[Local Irregularity Conjecture]\label{Con_nonTareColorable}For every colorable
connected graph $G$, it holds that $\chi_{%
%TCIMACRO{\TeXButton{TeX field}{\rm{irr}}}%
%BeginExpansion
\rm{irr}%
%EndExpansion
}^{\prime}(G)\leq3.$
\end{conjecture}

It was recently shown \cite{SedSkreMasna} that the Local Irregularity
Conjecture does not hold in general, since there exists a colorable cactus
graph, the so called bow-tie graph $B$ shown in Figure \ref{Fig_masna}, for
which $\chi_{%
%TCIMACRO{\TeXButton{TeX field}{\rm{irr}}}%
%BeginExpansion
\rm{irr}%
%EndExpansion
}^{\prime}(G)=4$. There, the following weaker version of Conjecture
\ref{Con_nonTareColorable} was proposed.

\begin{figure}[h]
\begin{center}
\includegraphics[scale=0.75]{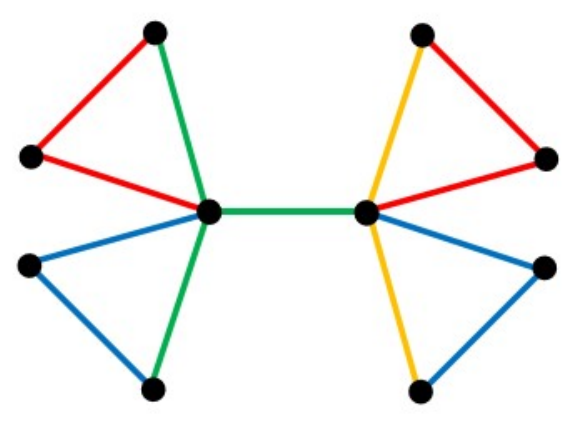}
\end{center}
\caption{The bow-tie graph $B$ and a locally irregular $4$-edge coloring of
it.}%
\label{Fig_masna}%
\end{figure}

\begin{conjecture}
\label{Con_weaker}Every colorable connected graph $G$ satisfies $\chi_{%
%TCIMACRO{\TeXButton{TeX field}{\rm{irr}}}%
%BeginExpansion
\rm{irr}%
%EndExpansion
}^{\prime}(G)\leq4.$
\end{conjecture}

Even though the graph $B$ contradicts Local Irregularity Conjecture, many
partial results support this conjecture. For example, the conjecture holds for
trees \cite{Baudon6}, unicyclic graphs and cacti with vertex disjoint cycles
\cite{SedSkreMasna}, graphs with minimum degree at least $10^{10}$
\cite{Przibilo}, $k$-regular graphs where $k\geq10^{7}$ \cite{Baudon5}. The
conjecture was also considered for general graphs, but there the upper bound
on $\chi_{%
%TCIMACRO{\TeXButton{TeX field}{\rm{irr}}}%
%BeginExpansion
\rm{irr}%
%EndExpansion
}^{\prime}(G)$ was first established to be $328$ \cite{Bensmail}, and then it
was lowered to $220$ \cite{Luzar}.

The results so far indicate that the cacti are relevant class of graphs for
the locally irregular edge colorings, since the non-colorable graphs are cacti
and the only known counterexample for the Local Irregularity Conjecture is
graph $B$ which is also a cactus graph. This motivated further investigation
of cacti, so in \cite{SedSkrePaper2, Yongtang} it was established that all
colorable cacti satisfy $\chi_{%
%TCIMACRO{\TeXButton{TeX field}{\rm{irr}}}%
%BeginExpansion
\rm{irr}%
%EndExpansion
}^{\prime}(G)\leq4$. In this paper we establish that $B$ is the only colorable
cactus graph with $\chi_{%
%TCIMACRO{\TeXButton{TeX field}{\rm{irr}}}%
%BeginExpansion
\rm{irr}%
%EndExpansion
}^{\prime}(G)=4.$ We conclude the paper conjecturing that $B$ is the only
counterexample to the Local Irregularity Conjecture in general. Throughout the
paper the proofs are omitted, since they are sometimes lengthy and technical,
and they are provided in the appendix at the end of the paper.

\section{Preliminaries}

Let us introduce some basic notions related to colorings, then several results
already known in literature mainly regarding trees and unicyclic graphs, and
then finally some basic notions regarding cacti.

The colors of an edge coloring will be denoted by letters $a,$ $b,$ $c,\ldots$
By $\phi_{a}(G)$ (resp. $\phi_{a,b}(G),$ $\phi_{a,b,c}(G)$) we will usually
denote an edge coloring of $G$ using one (resp. two, three) colors. For an
edge coloring $\phi_{a,b,c}$ of $G$ and a vertex $u\in V(G),$ by $\phi
_{a,b,c}(u)$ we will denote the set of colors incident to $u.$ If
$\phi_{a,b,c}(u)$ contains $k$ colors, then we say $u$ is $k$\emph{-chromatic}%
, specifically in the cases of $k=1$ and $k=2$ we say $u$ is
\emph{monochromatic} and \emph{bichromatic}, respectively. If $\phi_{a,b}$ is
an edge coloring of a graph $G$ which uses colors $a$ and $b,$ then
$\phi_{c,d}$ will denote the edge coloring of $G$ obtained from $\phi_{a,b}$
by replacing color $a$ with $c$ and color $b$ with $d$. In particular, the
edge coloring $\phi_{b,a}$ obtained from $\phi_{a,b}$ by swapping colors $a$
and $b$ will be the \emph{inversion} of $\phi_{a,b}.$

The $a$\emph{-degree} of a vertex $u\in V(G)$ is defined as the number of
edges incident to $u$ which are colored by $a$, and it is denoted by
$d_{G}^{a}(u).$ The same name and notation goes for any other color besides
$a$. Now, for a vertex $u\in V(G)$ with $k$ incident edges, say $uv_{1}%
,\ldots,uv_{k},$ colored by $a$, the $a$\emph{-sequence }is defined as
$d_{G}^{a}(v_{1}),\ldots,d_{G}^{a}(v_{k}).$ It is usually assumed that
vertices $v_{i}$ are denoted so that the $a$-sequence is non-increasing.

Let $G$ be a graph and $G_{0}$ a subgraph of $G.$ We say that a $k$-liec
$\phi$ of $G$ is an \emph{extension} of $k$-liec $\phi^{0}$ of $G_{0},$ if
$\phi(e)=\phi^{0}(e)$ for every $e\in E(G_{0}).$ We also say that we have
\emph{extended} $\phi^{0}$ to $G.$ We also need to introduce a \emph{sum of
colorings}, useful when combining two different colorings. Let $G_{i},$ for
$i=1,\ldots,k,$ be graphs with pairwise disjoint edge sets, and let $G$ be a
graph such that $E(G)=\cup_{i=1}^{k}E(G_{i}).$ If $\phi_{a,b,c}^{i}$ is an
edge coloring of the graph $G_{i}$ for $i=1,\ldots,k,$ then $\phi_{a,b,c}%
=\sum_{i=1}^{k}\phi_{a,b,c}^{i}$ will denote the edge coloring of $G$ such
that $e\in E(G_{i})$ implies $\phi_{a,b,c}(e)=\phi_{a,b,c}^{i}(e).$

Let $G$ be a colorable graph which admits a $k$-liec for $k\geq3.$ If $G$
contains a leaf $u,$ then a graph $G^{\prime}$ obtained from $G$ by appending
an even length path $P$ at $u$ also admits a $k$-liec. An \emph{ear} of a
graph $G$ is any subgraph $P$ of $G$ such that $P$ is a path and for any
internal vertex $u\in V(P)$ it holds that $d_{G}(u)=2.$ If $G$ contains an ear
$P_{q}=u_{0}u_{1}\ldots u_{q}$ of length $q\geq3,$ then at least one internal
vertex $u_{i}$ of $P_{q}$ is bichromatic by a $k$-liec of $G,$ so if
$G^{\prime}$ is obtained from $G$ by replacing an ear $P_{q}$ by an ear
$P_{q+2r}$ of the length $q+2r,$ then $G^{\prime}$ also admits a $k$-liec. In
both situations we say $G$ is obtained by \emph{trimming} $G^{\prime}$.

\begin{remark}
\label{Obs_trimmed}Let $G^{\prime}$ be a colorable graph which admits a
$k$-liec for $k\geq3.$ If $G$ is obtained by trimming $G^{\prime},$ then $G$
also admits a $k$-liec.
\end{remark}

A graph $G$ is \emph{totally trimmed} if it does not contain a pending path of
length $\geq3$ and an ear of length $\geq5.$ Because of Observation
\ref{Obs_trimmed}, in the rest of the paper we will tacitly assume every graph
$G$ is totally trimmed.

A tree rooted at its leaf will be called a \emph{shrub}. The \emph{root edge}
of a shrub $G$ is the only edge incident to the root vertex of $G.$ An edge
coloring of a shrub $G$ is said to be an \emph{almost locally irregular }%
$k$\emph{-edge coloring}, or $k$\emph{-aliec} for short, if either it is a
$k$-liec of $G$ or it is a coloring such that only the root edge is locally
regular. A \emph{proper }$k$\emph{-aliec} is any $k$-aliec which is not a
$k$-liec. Notice that in the case of a proper $k$-aliec, the root edge of the
shrub is an isolated edge of its color, i.e. it is not adjacent to any other
edge of the same color. Let us state few results regarding trees from
\cite{Baudon6}.

\begin{theorem}
\label{Tm_BaudonSchrub}Every shrub admits a $2$-aliec.
\end{theorem}

\begin{theorem}
\label{Tm_BaudonTree}For every colorable tree $T,$ it holds that $\chi_{%
%TCIMACRO{\TeXButton{TeX field}{\rm{irr}}}%
%BeginExpansion
\rm{irr}%
%EndExpansion
}^{\prime}(T)\leq3.$ Moreover, $\chi_{%
%TCIMACRO{\TeXButton{TeX field}{\rm{irr}}}%
%BeginExpansion
\rm{irr}%
%EndExpansion
}^{\prime}(T)\leq2$ if $\Delta(T)\geq5.$
\end{theorem}

According to Theorem \ref{Tm_BaudonSchrub}, a shrub admits a $2$-aliec which
will be denoted by $\phi_{a,b},$ where we assume the root edge of the shrub is
colored by $a$. Now, let $T$ be a tree, $u\in V(T)$ a vertex of maximum degree
in $T,$ and $v_{i}$ all the neighbors of $u$ for $i=1,\ldots,k.$ Denote by
$T_{i}$ a shrub of $T$ rooted at $u$ with the root edge $uv_{i}.$ A
\emph{shrub based} coloring of $T$ is defined by $\phi_{a,b}=\sum_{i=1}%
^{k}\phi_{a,b}^{i},$ where $\phi_{a,b}^{i}$ is an $2$-aliec of $T_{i}.$ Since
we assume that the root edge of $T_{i}$ is colored by $a$ in $\phi_{a,b}^{i},$
this implies $u$ is monochromatic in color $a$ by a shrub based coloring
$\phi_{a,b}$ of $T.$ Obviously, if a shrub based coloring $\phi_{a,b}$ is not
a liec of $T,$ only the edges incident to $u$ may be locally regular by
$\phi_{a,b}.$ Notice that $\binom{k}{2}$ different $2$-edge colorings of $T$
can be obtained from a shrub based coloring $\phi_{a,b}$ by swapping colors
$a$ and $b$ in some of the shrubs $T_{i}$. If none of those colorings is a
liec of $T,$ we say $\phi_{a,b}$ is \emph{inversion resistant}. The following
observation was established in \cite{Baudon6}.

\begin{remark}
\label{Observation_sequences}Let $T$ be a tree, $u\in V(T)$ a vertex of
maximum degree in $T$ and $\phi_{a,b}$ a shrub based coloring of $T$ rooted at
$u.$ The shrub based coloring $\phi_{a,b}$ will be inversion resistant in two
cases only:

\begin{itemize}
\item $d_{T}(u)=3$ and the $a$-sequence of $u$ by $\phi_{a,b}$ is $3,2,2$;

\item $d_{T}(u)=4$ and the $a$-sequence of $u$ by $\phi_{a,b}$ is $4,3,3,2$.
\end{itemize}
\end{remark}

The obvious consequence of the above observation is that in a tree $T$ with
$\chi_{%
%TCIMACRO{\TeXButton{TeX field}{\rm{irr}}}%
%BeginExpansion
\rm{irr}%
%EndExpansion
}^{\prime}(T)=3$ and $\Delta(T)=3$ (resp. $\Delta(T)=4$) the vertices of
degree $3$ (resp. $4$) must come in neighboring pairs. {\normalsize Also, if
$\chi_{%
%TCIMACRO{\TeXButton{TeX field}{\rm{irr}}}%
%BeginExpansion
\rm{irr}%
%EndExpansion
}^{\prime}(T)=3$, then a shrub based coloring of $T$ is inversion resistant. A
$3$-liec of such a tree $T$ is obtained from aliecs of shrubs in the following
way: if $d_{T}(u)=3$ then $\phi_{a,b,c}^{T}=\phi_{a,b}^{1}+\phi_{b,a}^{2}%
+\phi_{c,b}^{3}$ is a $3$-liec of $T,$ if $d_{T}(u)=4$ then $\phi_{a,b,c}%
^{T}=\phi_{a,b}^{1}+\phi_{a,b}^{2}+\phi_{b,a}^{3}+\phi_{c,b}^{4}$ is a
$3$-liec of $T.$ Notice that only the root vertex $u$ is $3$-chromatic by
$\phi_{a,b,c}^{T}$, all other vertices in $T$ are $1$- or $2$-chromatic.
Hence, we will call $u$ the \emph{rainbow root} of $\phi_{a,b,c}^{T}$.
Obviously, every vertex $u$ of maximum degree in a tree $T$ with $\chi_{%
%TCIMACRO{\TeXButton{TeX field}{\rm{irr}}}%
%BeginExpansion
\rm{irr}%
%EndExpansion
}^{\prime}(T)=3$ can be the rainbow root of a $3$-liec of $T$, because either
the degree sequence of $u$ by a shrub based coloring is $3,2,2$ (resp.
$4,3,3,2$) which means the above $3$-liecs can be constructed with $u$ being
the rainbow root, or the shrub based coloring would not be inversion
resistant, which implies $\chi_{%
%TCIMACRO{\TeXButton{TeX field}{\rm{irr}}}%
%BeginExpansion
\rm{irr}%
%EndExpansion
}^{\prime}(G)\leq2,$ a contradiction.\ }

{\normalsize Further, notice that the color $c$ is used by $\phi_{a,b,c}^{T}$
in precisely one shrub of $T$ and one can choose that one shrub to be any of
the shrubs of $T$ rooted at $u$. Let us now collect all this in the following
formal observation for further referencing. }

\begin{remark}
{\normalsize \label{Observation_maxDeg}Let $T$ be a colorable tree with
$\chi_{
%TCIMACRO{\TeXButton{TeX field}{\rm{irr}} }%
%BeginExpansion
\rm{irr}
%EndExpansion
}^{\prime}(T)=3.$ Then all vertices of maximum degree in $T$ come in
neighboring pairs. Also, every vertex of maximum degree in $T$ can be the
rainbow root of a $3$-liec of $T$. Finally, for any vertex $u$ of maximum
degree in $T$ and for any shrub $T_{i}$ of $T$ rooted at $u,$ there exists a
$3$-liec of $T$ such that the color $c$ is used only in $T_{i}$. }
\end{remark}

{\normalsize We will also need the following result from \cite{SedSkreMasna}%
\ on unicyclic graphs. }

\begin{theorem}
{\normalsize \label{Tm_unicyclic}Let $G$ be a colorable unicyclic graph. Then
$\chi_{%
%TCIMACRO{\TeXButton{TeX field}{\rm{irr}}}%
%BeginExpansion
\rm{irr}%
%EndExpansion
}^{\prime}(G)\leq3.$}
\end{theorem}

Now, let us introduce few notions regarding cacti. A \emph{grape} $G$ is any
cactus graph with at least one cycle in which all cycles share a vertex $u,$
and the vertex $u$ is called the \emph{root} of $G.$ A \emph{berry} $G_{i}$ of
a grape $G$ is any subgraph of $G$ induced by $V(G_{i}^{\prime})\cup\{u\},$
where $u$ is the root of $G$ and $G_{i}^{\prime}$ a connected component of
$G-u.$ Notice that a berry $G_{i}$ can be either a unicyclic graph in which
$u$ is of degree $2$ or a tree in which $u$ is a leaf, so such berries will be
called \emph{unicyclic berries} and \emph{acyclic berries}, respectively. For
convenience, in a unicyclic berry $G$ we will say a vertex $v$ is
\emph{acyclic} if it does not belong to the cycle of $G,$ otherwise we call it
\emph{cyclic}.

\begin{figure}[h]
\begin{center}
\includegraphics[scale=0.85]{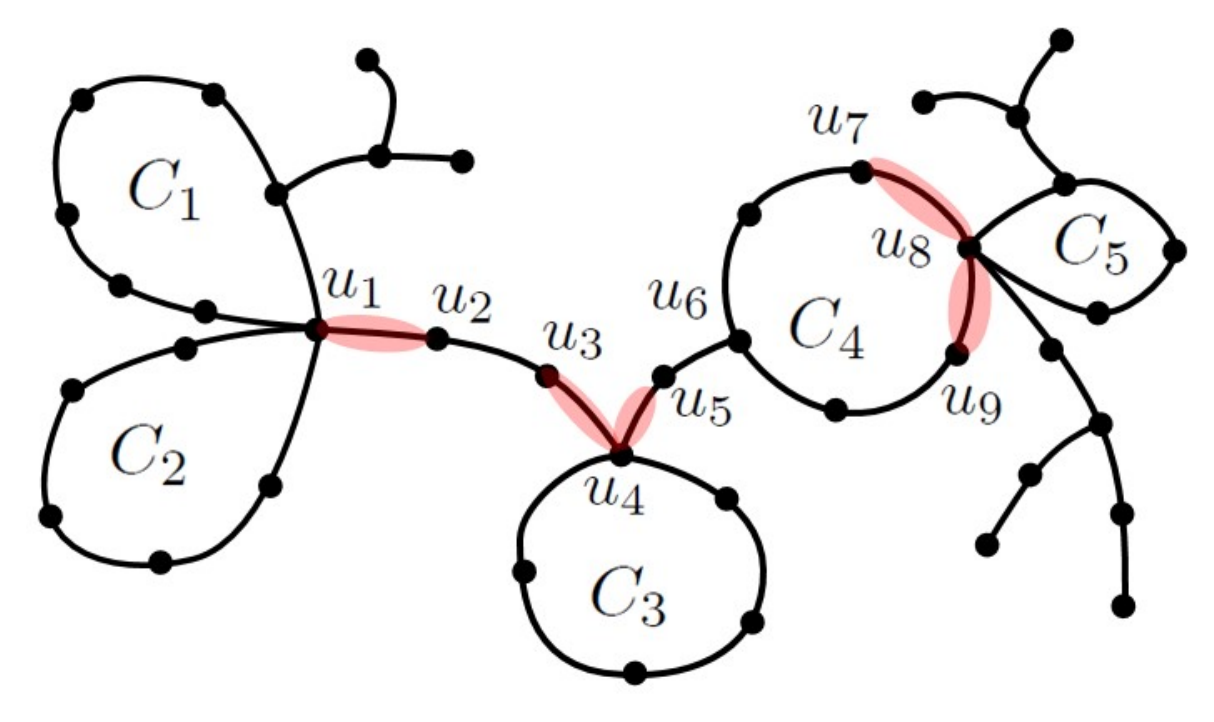}
\end{center}
\caption{A cactus graph $G$ with five cycles which contains two end-grapes,
$G_{u_{1}}$ and $G_{u_{8}}.$ The end-grape $G_{u_{1}}$ has two unicyclic
berries and one exit edge $u_{1}u_{2}.$ The end-grape $G_{u_{8}}$ consists of
one unicyclic and one acyclic berry and it has two exit edges $u_{8}u_{7}$ and
$u_{8}u_{9}.$ Notice that the cycle $C_{3}$ is not an end-grape of $G$ since
it has two exit edges $u_{4}u_{3}$ and $u_{4}u_{5}$ and they do not belong to
a same cycle.}%
\label{Fig_endCycles}%
\end{figure}

Let $G$ be a cactus graph with $\mathfrak{c}\geq2$ cycles which is not a
grape. An \emph{end-grape} $G_{u}$ of $G$ is a subgraph of $G$ such that:

\begin{itemize}
\item $G_{u}$ is a grape rooted at $u$ and $u$ is the only vertex of $G_{u}$
incident to edges from $G-E(G_{u}),$ and

\item $u$ is incident to either one such edge or two such edges which then
must belong to a same cycle of $G-E(G_{u}),$ and such edges are called the
\emph{exit edges} of $G_{u}$.
\end{itemize}

\noindent This notion is illustrated by Figure \ref{Fig_endCycles}. If $G_{u}$
is an end-grape of a cactus graph $G,$ then we denote $G_{0}=G-E(G_{u})$ and
$G_{0}$ will be called the \emph{root component} of $G_{u}$.

\section{Berry colorings}

In a cactus graph, we will color the berries of an end-grape by three kinds of
colorings: primary, secondary, and tertiary. So, let us define such colorings
and let us establish which berries admit such colorings.

\paragraph{Primary coloring.}

Let us now introduce two kinds of a primary coloring which will mainly be used
when coloring berries of an end-grape.

\begin{definition}
Let $G$ be a graph rooted at a vertex $u$ with $d_{G}(u)\in\{1,2\}.$ Let $v$
be a neighbor of $u$ and $w$ the other neighbor of $u$ if $d_{G}(u)=2.$ A
\emph{standard primary coloring} of $G$ is any $3$-edge coloring $\phi
_{a,b,c}$ of $G$ with the following properties:

\begin{itemize}
\item every edge non-incident to $u$ is locally irregular by $\phi_{a,b,c}$;

\item $u$ is monochromatic by $\phi_{a,b,c},$ say $\phi_{a,b,c}(u)=\{c\}$;

\item $d^{c}(v)\leq2$ and $d^{c}(w)\leq2$.
\end{itemize}
\end{definition}

It would be convenient to use a standard primary coloring for all berries, but
some berries do not admit such a coloring. Thus, we introduce an alternative
primary coloring which we will use in case of such berries.

\begin{definition}
Let $G$ be a graph rooted at a vertex $u$ with $d_{G}(u)\in\{1,2\}.$ Let $v$
be a neighbor of $u$ and $w$ the other neighbor of $u$ if $d_{G}(u)=2.$

\begin{itemize}
\item If $d_{G}(u)=1,$ an \emph{alternative primary coloring} of $G$ is any
$3$-liec $\phi_{a,b,c}$ of $G$ such that $v$ is monochromatic by $\phi
_{a,b,c},$ say $\phi_{a,b,c}(v)=\{c\}.$

\item If $d_{G}(u)=2,$ an \emph{alternative primary coloring} of $G$ is any
$3$-edge coloring $\phi_{a,b,c}$ of $G$ with the following properties:

\begin{itemize}
\item every edge $e\not =uv$ is locally irregular by $\phi_{a,b,c}$;

\item $u$ is bichromatic by $\phi_{a,b,c},$ say $\phi_{a,b,c}(uv)=c$ and
$\phi_{a,b,c}(uw)=a$;

\item $d^{c}(v)\leq2$ and $d^{a}(w)\in\{2,4\}$;

\item if $d^{c}(v)=2$ then $d^{a}(w)=4$.
\end{itemize}
\end{itemize}
\end{definition}

Let us now define seven types of berries for which it is not convenient for us
to use a standard primary coloring.

\begin{figure}[h]
\begin{center}%
\begin{tabular}
[c]{c}%
\begin{tabular}
[c]{llllll}%
$B_{1}$ & \includegraphics[scale=0.6]{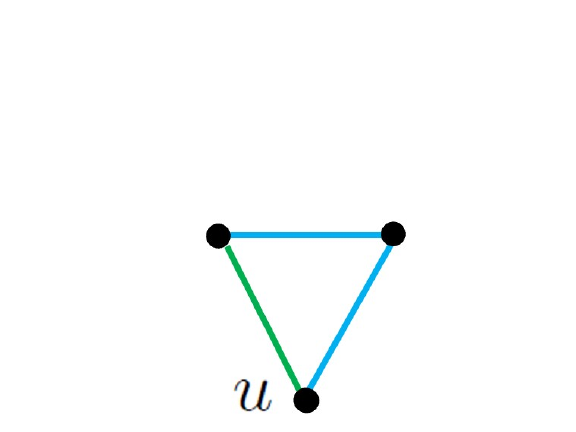} & $B_{2}$ &
\includegraphics[scale=0.6]{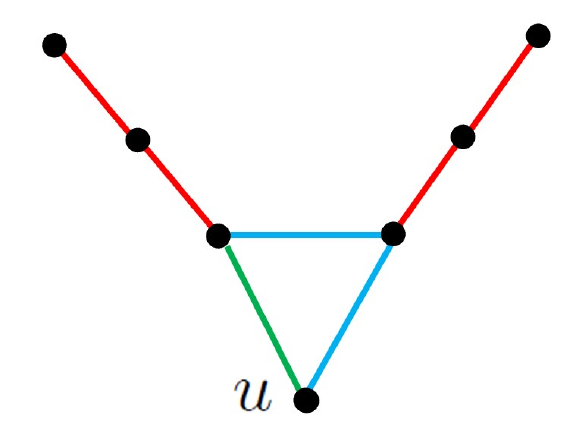} & $B_{3}$ &
\includegraphics[scale=0.6]{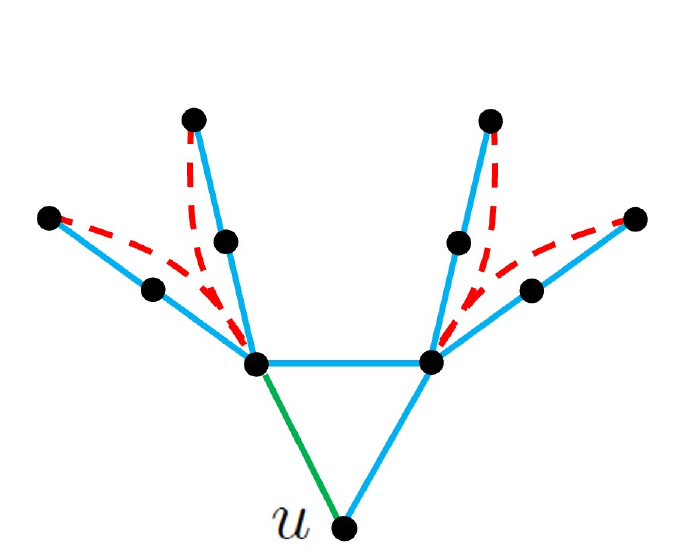}
\end{tabular}
\\%
\begin{tabular}
[c]{llll}%
$B_{4}$ & \includegraphics[scale=0.6]{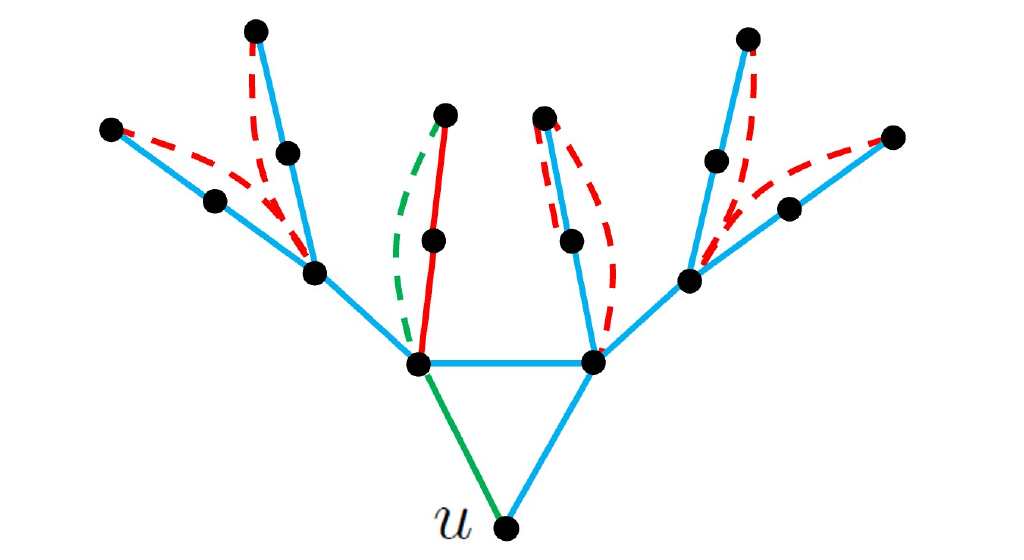} & $B_{5}$ &
\includegraphics[scale=0.6]{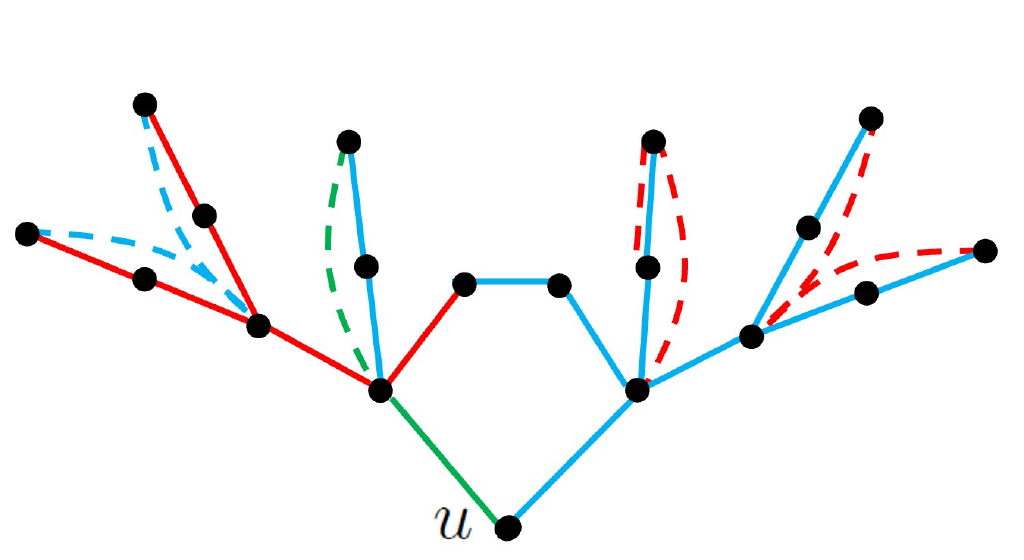}
\end{tabular}
\\%
\begin{tabular}
[c]{llll}%
$B_{6}$ & \includegraphics[scale=0.6]{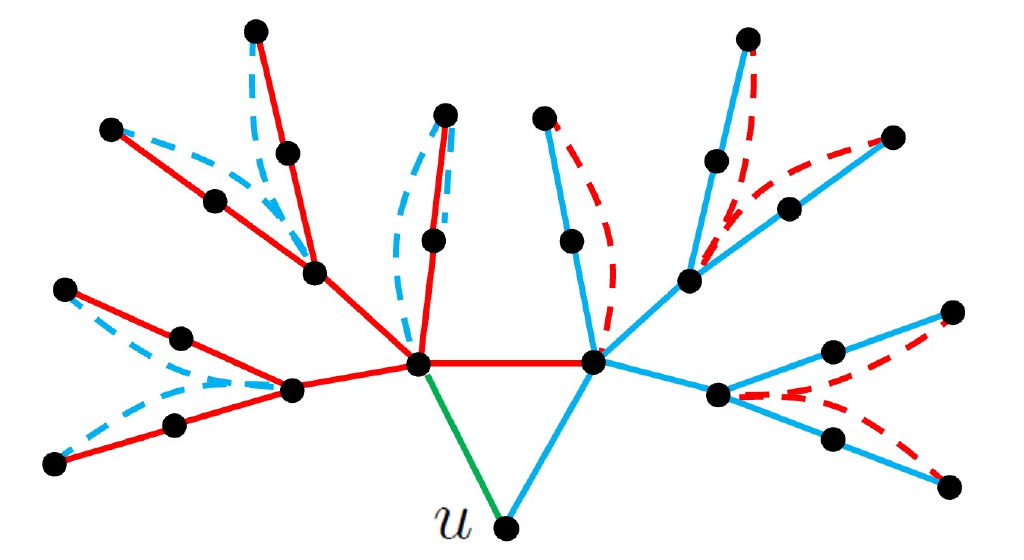} & $B_{7}$ &
\includegraphics[scale=0.6]{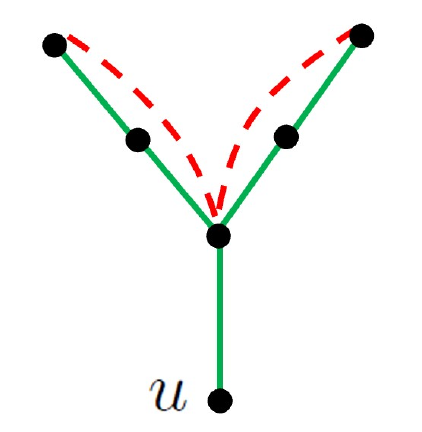}
\end{tabular}
\end{tabular}
\end{center}
\caption{Alternative berries $B_{1},\ldots,B_{7}$ rooted at $u.$ Dashed edges
are not included in the berry itself, but only in its closure (which will be
defined later).}%
\label{Fig_alternativeBerries}%
\end{figure}

\begin{definition}
The berries $B_{1},\ldots,B_{7}$ shown in Figure \ref{Fig_alternativeBerries}
(ignoring the dashed lines) will be called \emph{alternative berries}. All
other berries will be called \emph{standard berries}.
\end{definition}

A unicyclic berry consisting of the triangle $uvw$ and an even length path
hanging only at $w$ will also be considered to be $B_{2}$. The following
proposition connects berries and the primary coloring.

\begin{proposition}
\label{Prop_primaryColoring}Every standard berry admits a standard primary
coloring, every alternative berry admits an alternative primary coloring.
\end{proposition}

\begin{proof}
An alternative primary coloring of alternative berries is shown in Figure
\ref{Fig_alternativeBerries} (without the dashed lines), and the rest of the
proof is in the appendix.
\end{proof}

In the sequel we will say shortly 'primary coloring', assuming standard
primary coloring for standard berries and alternative primary coloring for
alternative berries.

Notice that the above proposition does not claim that alternative berries do
not admit a standard primary coloring. Namely, it is easily verified that
berries $B_{1},$ $B_{2}$ and $B_{3}$ indeed do not admit a standard primary
coloring, but berries $B_{4},$ $B_{5},$ $B_{6}$ and $B_{7}$ do. The reason why
these four berries are included among alternative berries is that their
closure (which will be introduced later) does not admit a standard primary
coloring. To conclude, these berries for which we established the existence of
a primary coloring will be called the \emph{primary colored} berries.

\paragraph{Secondary coloring.}

Secondary coloring will be needed for some standard berries. Let us formally
define such a coloring.

\begin{definition}
Let $G\not =B_{7}$ be a graph rooted at a vertex $u$ with $d_{G}%
(u)\in\{1,2\}.$ Let $v$ be a neighbor of $u$ and $w$ the other neighbor of $u$
if $d_{G}(u)=2.$

\begin{itemize}
\item If $d_{G}(u)=1,$ a\emph{ secondary coloring} of $G$ is any $3$-edge
coloring $\phi_{a,b,c}$ of $G$ such that all edges not incident to $u$ are
locally irregular and $d^{c}(v)\not =2$ where $c$ is the color of $uv$ by
$\phi_{a,b,c}.$

\item If $d_{G}(u)=2,$ a \emph{secondary coloring} of $G$ is any $3$-edge
coloring $\phi_{a,b,c}$ of $G$ with the following properties:

\begin{itemize}
\item every edge $e\not =uv$ is locally irregular by $\phi_{a,b,c}$;

\item $u$ is bichromatic by $\phi_{a,b,c},$ say $\phi_{a,b,c}(uv)=c$ and
$\phi_{a,b,c}(uw)=a$;

\item $d^{c}(v)\leq2$;

\item if $d^{c}(v)=2$ then $d^{a}(w)\geq3.$
\end{itemize}
\end{itemize}
\end{definition}

We exclude $B_{7}$ from the above definition, since we always want an
alternative primary coloring of it. In the following two propositions we
establish for some berries that they admit a secondary coloring.

\begin{proposition}
\label{Prop_secondaryColoring1}Let $G$ be a standard berry rooted at $u$ with
$d_{G}(u)=1.$ If $G$ is not an even length path, then $G$ admits a secondary coloring.
\end{proposition}

\begin{proposition}
\label{Prop_secondaryColoring2}Let $G$ be a standard berry rooted at $u$ with
$d_{G}(u)=2.$ If $G$ does not admit a primary coloring which is a liec, then
$G$ admits a secondary coloring.
\end{proposition}

Graphs for which a secondary coloring is constructed in Propositions
\ref{Prop_secondaryColoring1} and \ref{Prop_secondaryColoring2} will be called
the \emph{secondary colored} graphs.

\begin{figure}[h]
\begin{center}%
\begin{tabular}
[c]{llll}%
$B_{2}$ & \includegraphics[scale=0.6]{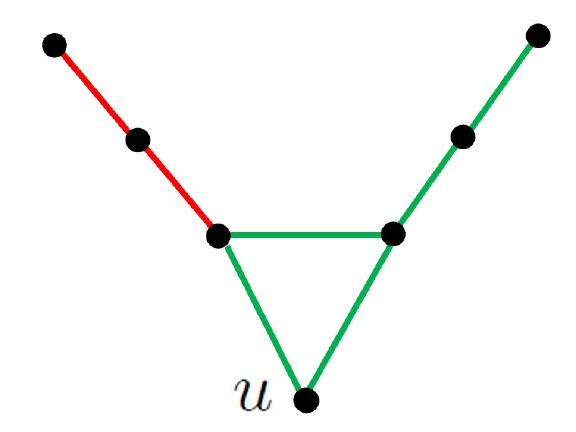} & $B_{3}$ &
\includegraphics[scale=0.6]{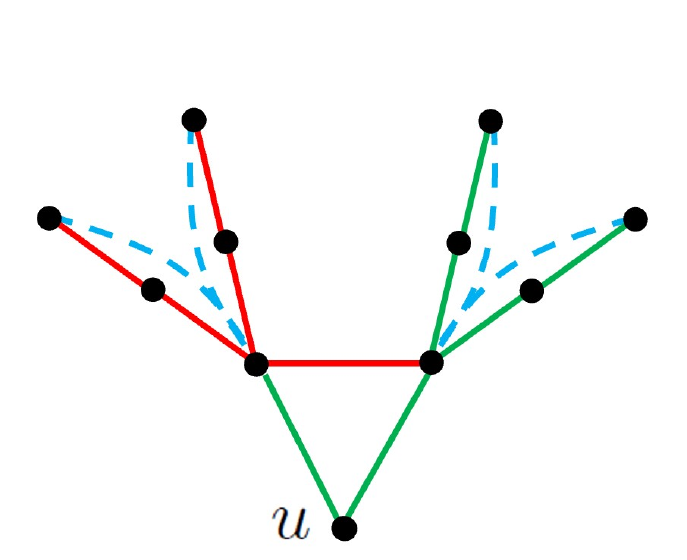}\\
$B_{4}$ & \includegraphics[scale=0.6]{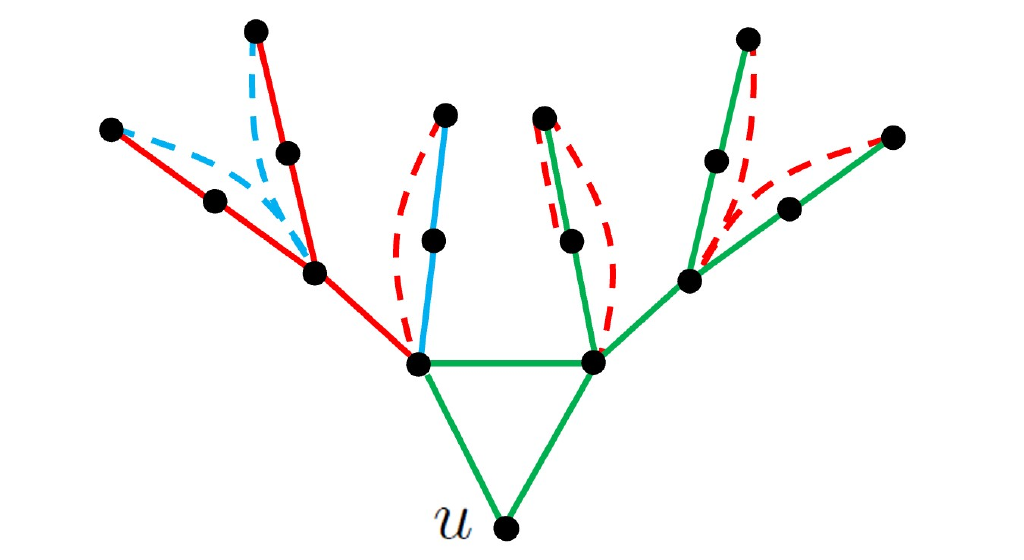} & $B_{5}$ &
\includegraphics[scale=0.6]{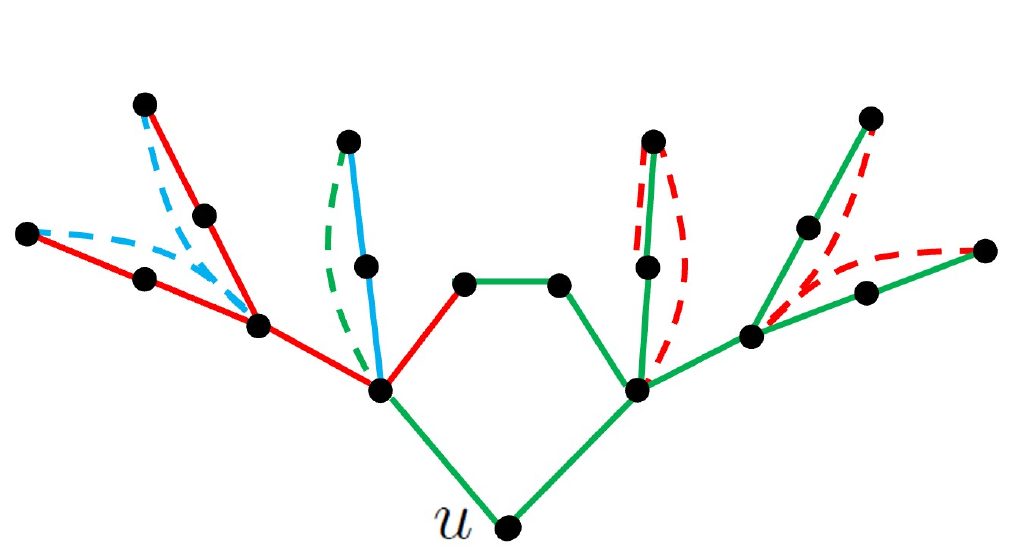}\\
$B_{6}$ & \multicolumn{3}{l}{\includegraphics[scale=0.6]{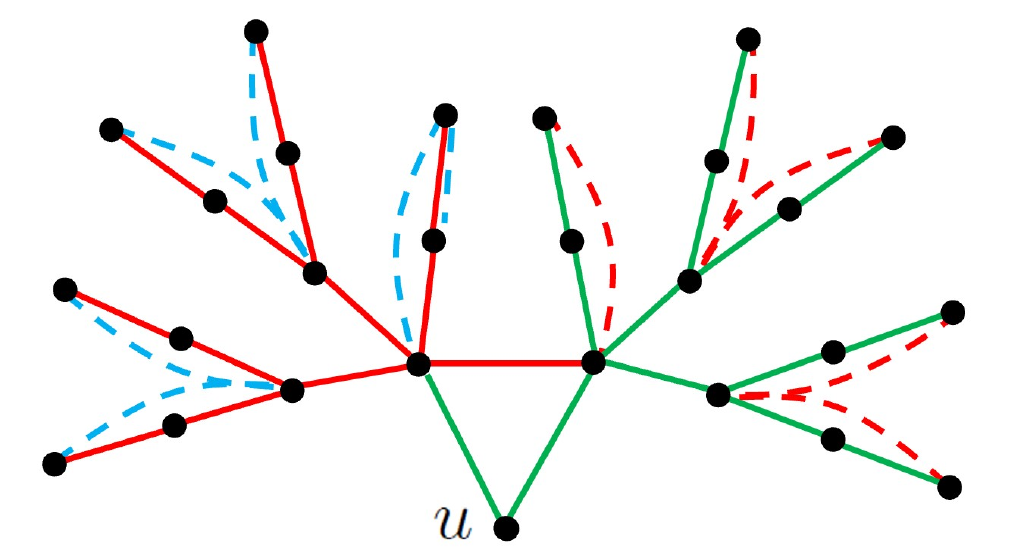}}%
\end{tabular}
\end{center}
\caption{A tertiary coloring of berries $B_{2},\ldots,B_{6}$ and its
closures.}%
\label{Fig_tertiaryCol}%
\end{figure}

\paragraph{Tertiary coloring.}

A third variant of a berry coloring will also be needed, a so called tertiary coloring.

\begin{definition}
Let $G$ be a graph rooted at a vertex $u$ with $d_{G}(u)=2.$ Let $v$ and $w$
be the neighbors of $u.$ A \emph{tertiary coloring} of $G$ is any $3$-edge
coloring $\phi_{a,b,c}$ of $G$ with the following properties:

\begin{itemize}
\item every edge $e\not =uv$ is locally irregular by $\phi_{a,b,c}$;

\item $u$ is monochromatic by $\phi_{a,b,c},$ say $\phi_{a,b,c}(u)=\{c\}$;

\item $d^{c}(v)\leq2$ and $d^{c}(w)\in\{3,4\}.$
\end{itemize}
\end{definition}

A tertiary coloring of $B_{2},\ldots,B_{6}$ is shown in Figure
\ref{Fig_tertiaryCol} (without dashed lines), so the following proposition
holds and such berries will be called \emph{tertiary colored} berries.

\begin{proposition}
\label{Prop_tertiaryColoring}Let $G$ be an alternative berry. If $G\not =%
B_{1},B_{7},$ then $G$ admits a tertiary coloring.
\end{proposition}

\section{Grape and end-grape colorings}

Let $G$ be a cactus graph and $G_{u}$ an end-grape of $G$ with a root
component $G_{0}.$ The main idea of the proof is to establish that every
coloring of $G_{0}$ can be extended to $G.$ Thus the induction on the number
of cycles in a cactus graph, with Theorem \ref{Tm_unicyclic} as basis, would
yield a $3$-liec of all cacti. We already know that this will not go smoothly,
because of the bow-tie graph $B,$ but in the end it will work for all other cacti.

\begin{proposition}
\label{Prop_grapes1}Let $G$ be a colorable cactus graph with $\mathfrak{c}%
\geq2$ cycles which is a grape. Then $G$ admits a $3$-liec.
\end{proposition}

Next, we will focus our attention on cactus graphs that are not grapes, so
they must contain an end-grape $G_{u}$ with a root component $G_{0}.$ Recall
that $d_{G_{0}}(u)\in\{1,2\},$ so the neighbors of $u$ in $G_{0}$ will be
denoted by $u_{1}$ and $u_{2},$ where $u_{2}$ exists only in the case
$d_{G_{0}}(u)=2.$ We will also need a modification $G_{0}^{\prime}$ of a root
component $G_{0}$ defined by $G_{0}^{\prime}=G_{0}+uv_{1}$ where $v_{1}$
denotes a neighbor of $u$ which belongs to a unicyclic berry of $G_{u}.$ Now,
let us define the following six particular end-grapes.

\begin{definition}
The end-grapes $A_{1},\ldots,A_{6}$ shown in Figure \ref{Fig_singularGrapes}
are called \emph{singular} end-grapes. All other end-grapes are called
\emph{regular} end-grapes.
\end{definition}

\begin{figure}[h]
\begin{center}%
\begin{tabular}
[c]{llll}%
$A_{1}$ & \includegraphics[scale=0.65]{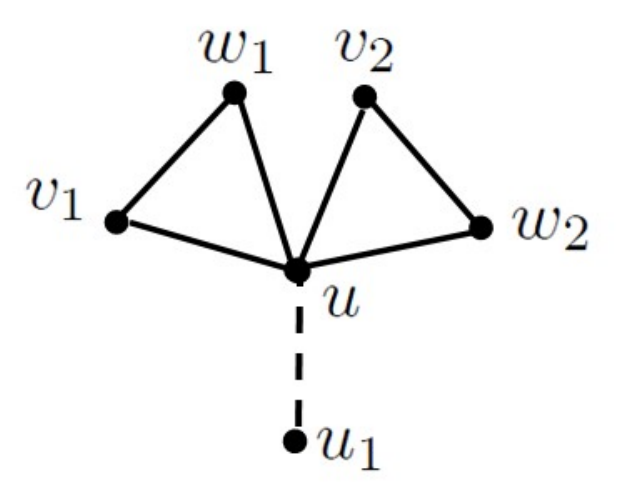} & $A_{2}$ &
\includegraphics[scale=0.65]{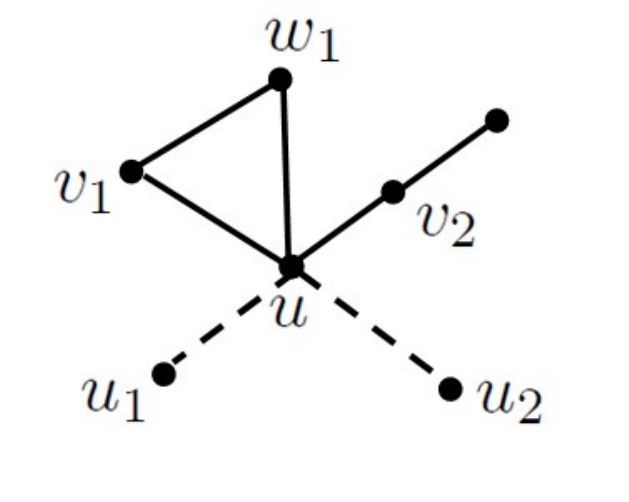}\\
$A_{3}$ & \includegraphics[scale=0.65]{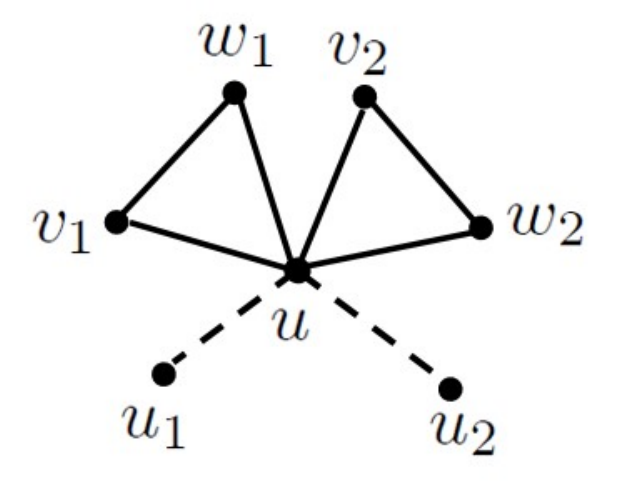} & $A_{4}$ &
\includegraphics[scale=0.65]{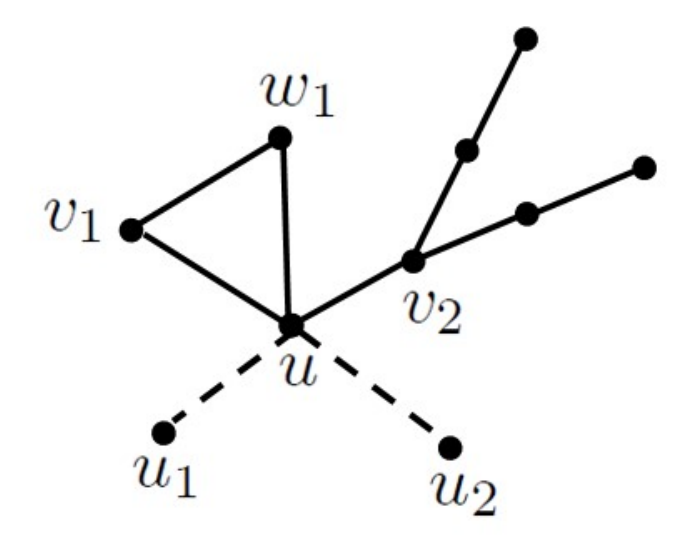}\\
$A_{5}$ & \includegraphics[scale=0.65]{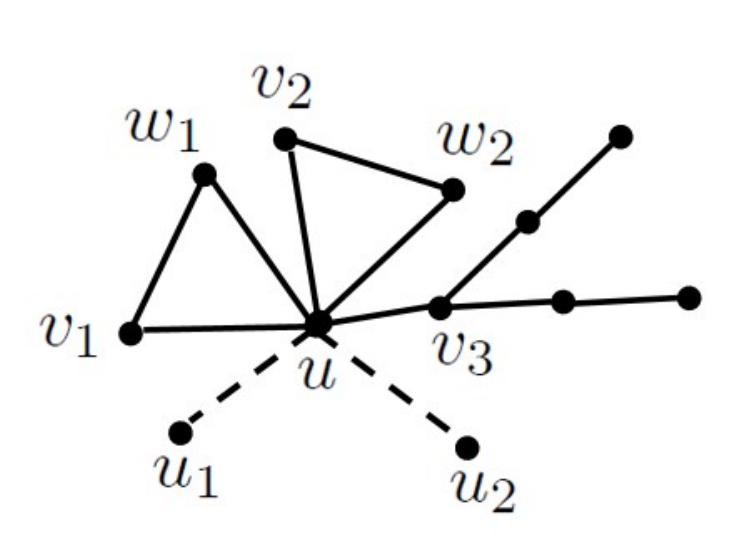} & $A_{6}$ &
\includegraphics[scale=0.65]{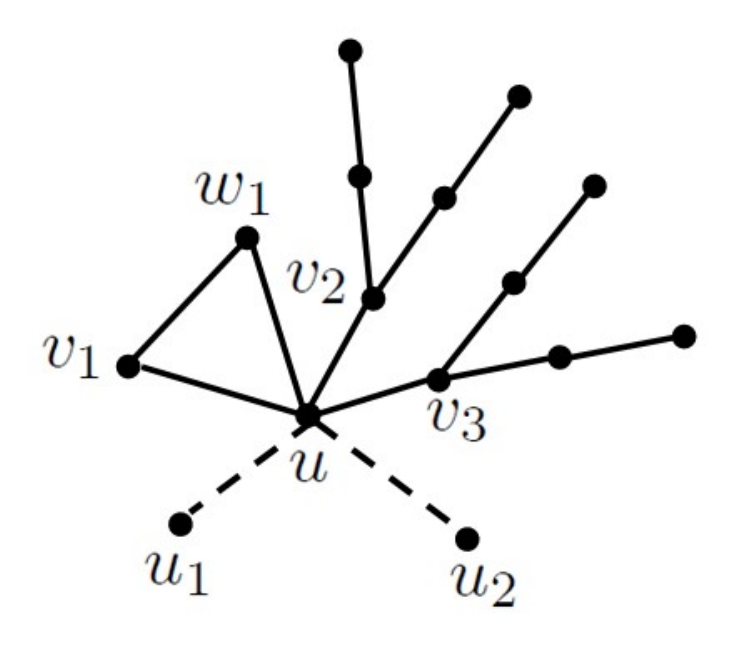}
\end{tabular}
\end{center}
\caption{The singular end-grapes $A_{1},\ldots,A_{6}.$}%
\label{Fig_singularGrapes}%
\end{figure}

The attempt to reduce a coloring of $G$ to a coloring of $G_{0}$ yields the
following proposition.

\begin{proposition}
\label{Prop_reduction1}Let $G$ be a cactus graph with $\mathfrak{c}\geq2$
cycles which is not a grape. If $G$ contains an end-grape $G_{u}$ with a root
component $G_{0}$ such that $G_{u}\not =A_{1},A_{2},A_{3},$ then there exists
a colorable graph $\tilde{G}_{0}\in\{G_{0},G_{0}^{\prime}\}$ such that every
$k$-liec of $\tilde{G}_{0}$, for $k\geq3$, can be extended to a $k$-liec of
$G.$
\end{proposition}

In the above proposition, we must exclude end-grapes $A_{1},A_{2},A_{3}$,
since some of the colorings of $\tilde{G}_{0}$ cannot be extended to such an
end-grape. We address this problem by removing an edge $uw_{j}$ incident to
$u$ from every triangle of $A_{i}$ ($i=1,2,3$). Thus an end-grape $A_{i}$
becomes an acyclic structure in a spanning subgraph of $G.$

In order to be more precise, we need the following notions. Let $G$ be a berry
rooted at $u$ and $x\in V(G)$ a vertex distinct from $u.$ A \emph{tail}
$T_{x}$ of the vertex $x$ is a subgraph of $G$ rooted at $x$ defined in a
following way:

\begin{itemize}
\item if $x$ belongs to the cycle $C$ of $G,$ then $T_{x}$ is the connected
component of $G-E(C)$ which contains $x;$

\item if $x$ does not belong to the cycle of $G,$ then $T_{x}$ is the
connected component of $G-xy$ which contains $x,$ where $y$ is the neighbor of
$x$ such that $d(u,y)<d(u,x).$
\end{itemize}

\noindent We also say that the tail $T_{x}$ is \emph{hanging} at $x$. Now, an
\emph{opening} $A_{i}^{\mathrm{{op}}}$ for $i\in\{1,\ldots,6\}$ is the tree
rooted at $u$ which is obtained from $A_{i}$ by removing an edge $uw_{j}$ from
every triangle of $A_{i}.$ A vertex $x\not =u$ is said to be a \emph{docking
vertex} of a berry $G$ if $T_{x}=A_{i}^{\mathrm{{op}}}$ for $i\in
\{1,\ldots,6\}.$ Let us denote by $D$ the set of all docking vertices of $G.$
A \emph{closure} $G_{X}^{\mathrm{{cl}}}$ of a berry $G$ rooted at $u$ is a
graph obtained by replacing $A_{i}^{\mathrm{{op}}}$ by $A_{i}$ for every
docking vertex $x\in X\subseteq D.$ Notice that within $A_{i}^{\mathrm{{op}}%
},$ for $i\in\{4,5,6\},$ a docking vertex for $A_{1}^{\mathrm{{op}}}$ is
contained which can also be closed in $G_{X}^{\mathrm{{cl}}}.$ We will usually
omit subscript $X$ when it does not lead to confusion. Finally, an
\emph{opening} $G^{\mathrm{{op}}}$ of a cactus graph $G$ is a spanning
subgraph of $G$ obtained by opening all end-grapes $A_{i}$ of $G$ for
$i\in\{1,\ldots,6\}.$ Notice that opening of an end-grape $A_{i}$ creates an
acyclic structure $A_{i}^{\mathrm{{op}}}$ hanging at a vertex of
$G^{\mathrm{{op}}}.$ An \emph{internal end-grape} of $G$ is any end-grape of
$G^{\mathrm{{op}}}.$ Notice that $G$ may not have an internal end-grape, since
$G^{\mathrm{{op}}}$ can be a tree, a unicyclic graph or a grape.

Now, let $G$ be a cactus graph in which every end-grape is $A_{1}$, $A_{2}$ or
$A_{3}.$ Notice that $G^{\mathrm{{op}}}$ cannot contain end-grapes $A_{1}$,
$A_{2}$ or $A_{3},$ since any such end-grape of $G^{\mathrm{{op}}}$ would also
be an end-grape of $G$ and thus opened in $G^{\mathrm{{op}}}.$ Notice that
$G^{\mathrm{{op}}}$ may be a tree, a unicyclic graph, a grape with
$\mathfrak{c}\geq2$ cycles or a cactus graph with $\mathfrak{c}\geq2$ cycles
which is not a grape, so $G^{\mathrm{{op}}}$ admits a $3$-liec according to
Theorems \ref{Tm_BaudonTree}, \ref{Tm_unicyclic}, Propositions
\ref{Prop_grapes1}, \ref{Prop_reduction1}, respectively, but it may not be
possible to extend it to a $3$-liec of $G.$

The case when $G^{\mathrm{{op}}}$ is a tree or unicyclic graph is easily
resolved, as we will do in the proof of main Theorem \ref{Tm_main}, but now we
need to introduce some auxiliary results for the case of $G^{\mathrm{{op}}}$
containing $\mathfrak{c}\geq2$ cycles. If $G^{\mathrm{{op}}}$ is a grape, then
it consists of berries. If $G^{\mathrm{{op}}}$ is not a grape, then we focus
on precisely one end-grape $G_{u}$ of $G^{\mathrm{{op}}}$ which also consists
of berries. Recall that $G_{u}$ is an internal end-grape of $G,$ and we define
$G^{\mathrm{{lop}}}$ to be a graph obtained from $G$ by opening end-grapes
only locally, i.e. only those end-grapes of $G$ whose root vertex belongs to
$G_{u}$. Notice that $G_{u}$ is an end-grape of $G^{\mathrm{{lop}}}$ and
$G_{u}\not =A_{1},A_{2},A_{3}.$

If $G^{\mathrm{{op}}}$ with $\mathfrak{c}\geq2$ cycles is (resp. is not) a
grape, then a $3$-liec of $G^{\mathrm{{op}}}$ (resp. $G^{\mathrm{{lop}}}$)
will be extendable to a $3$-liec of $G$ if the following two claims hold:

\begin{itemize}
\item[(i)] every berry of $G^{\mathrm{{op}}}$ (resp. of end-grape $G_{u}$ in
$G^{\mathrm{{lop}}}$) is colored by a primary, secondary or tertiary coloring;

\item[(ii)] a primary, secondary or tertiary coloring of every berry can be
extended to any of its closures.
\end{itemize}

\noindent The claim (ii) indeed holds, as is established in the following proposition.

\begin{proposition}
\label{Prop_ColoringExtendableClosure1}Let $G$ be a berry rooted at $u$ and
$G^{\mathrm{{cl}}}$ a closure of $G.$ If $G$ is colored by a primary (resp.
secondary, tertiary) coloring, then $G^{\mathrm{{cl}}}$ can also be colored by
a primary (resp. secondary, tertiary) coloring such that every berry of every
end-grape of $G^{\mathrm{{cl}}}$ is colored by a primary, secondary or
tertiary coloring.
\end{proposition}

\begin{figure}[h]
\begin{center}
\includegraphics[scale=0.75]{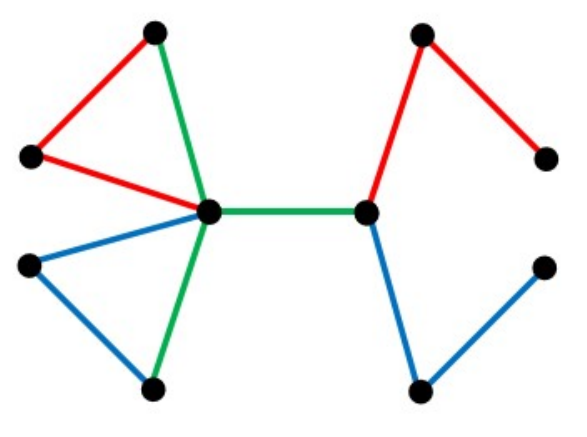}
\end{center}
\caption{A graph $B^{\prime}$ and a $3$-liec of it.}%
\label{Fig_Bprime}%
\end{figure}

As for the condition (i), it does not hold for the grape $B^{\prime}$ shown in
Figure \ref{Fig_Bprime}, since $B^{\prime}$ does admit a $3$-liec, but not
such that every berry is colored by a primary, secondary or tertiary coloring.
For all other grapes it does hold, which is established in the following proposition.

\begin{proposition}
\label{Prop_grapes2}Let $G$ be a cactus graph with $\mathfrak{c}\geq2$ cycles
which is a grape. If $G\not =B^{\prime},$ then $G$ admits a $3$-liec in which
every berry is colored by a primary, secondary or tertiary coloring.
\end{proposition}

In the case of $G^{\mathrm{{lop}}}$ which is not a grape and thus contains an
end-grape $G_{u}$, the claim (i) again does not hold for all end-grapes
$G_{u}$, since end-grapes $A_{4},A_{5},A_{6}$ arise as new exceptions.

\begin{proposition}
\label{Prop_reduction2}Let $G$ be a cactus graph with $\mathfrak{c}\geq2$
cycles which is not a grape and which does not contain end-grapes $A_{1}%
,A_{2},A_{3}$. If $G$ contains an end-grape $G_{u}$ with a root component
$G_{0}$ such that $G_{u}\not =A_{4},A_{5},A_{6},$ then there exists a
colorable graph $\tilde{G}_{0}\in\{G_{0},G_{0}^{\prime}\}$ such that every
$k$-liec of $\tilde{G}_{0},$ for $k\geq3$, can be extended to a $k$-liec of
$G$ in which every berry of $G_{u}$ is colored by a primary, secondary or
tertiary coloring.
\end{proposition}

In order to avoid the new exceptions $A_{4},$ $A_{5}$ and $A_{6},$ we repeat
the same procedure as with $A_{1},$ $A_{2}$ and $A_{3}$. Namely, if $G$
contains only end-grapes $A_{1}$, $A_{2}$, $A_{3}$, and $G^{\mathrm{{op}}}$
contains only end-grapes $A_{4}$, $A_{5}$, $A_{6},$ then we consider the graph
$(G^{\mathrm{{op}}})^{\mathrm{{op}}}.$ Again, $(G^{\mathrm{{op}}%
})^{\mathrm{{op}}}$ may be a tree, a unicyclic graphs, or a cactus graph with
$\mathfrak{c}\geq2$ cycles which may or may not be a grape. Here, we need an
auxiliary result only for the last case when $(G^{\mathrm{{op}}}%
)^{\mathrm{{op}}}$ has $\mathfrak{c}\geq2$ cycles and is not a grape. In that
case, let $G_{u}$ be an end-grape of $(G^{\mathrm{{op}}})^{\mathrm{{op}}}$ and
let $(G^{\mathrm{{lop}}})^{\mathrm{{lop}}}$ be the graph obtained from $G$ by
opening only those end-grapes of $G$ and then $G^{\mathrm{{op}}}$ whose root
vertex is in $G_{u}.$ Notice that $G_{u}$ is an end-grape of
$(G^{\mathrm{{lop}}})^{\mathrm{{lop}}}$ and $G_{u}\not =A_{i}$ for
$i\in\{1,\ldots,6\}.$ According to Proposition
\ref{Prop_ColoringExtendableClosure1}, a closure of every berry of $G_{u}$
admits a primary, secondary or a tertiary coloring, so it only remains to
establish that $(G^{\mathrm{{lop}}})^{\mathrm{{lop}}}$ admits a $3$-liec such
that every berry of $G_{u}$ is colored by a primary, secondary or tertiary coloring.

\begin{proposition}
\label{Prop_reduction3}Let $G$ be a cactus graph with $\mathfrak{c}\geq2$
cycles which is not a grape and which does not contain end-grapes
$A_{1},\ldots,A_{6}$. For an end-grape $G_{u}$ of $G$ with a root component
$G_{0},$ there exists a colorable graph $\tilde{G}_{0}\in\{G_{0},G_{0}%
^{\prime}\}$ such that every $k$-liec of $\tilde{G}_{0},$ for $k\geq3,$ can be
extended to $k$-liec of $(G^{\mathrm{{op}}})^{\mathrm{{op}}}$ in which every
berry of $G_{u}$ is colored by a primary, secondary or tertiary coloring.
\end{proposition}

Now, we can state the main theorem of the paper.

\begin{theorem}
\label{Tm_main}Let $G\not =B,B^{\prime}$ be a colorable cactus graph with
$\mathfrak{c}\geq2$ cycles. Then there exists a $3$-liec of $G$ such that
every berry of every end-grape is colored by a primary, secondary or tertiary coloring.
\end{theorem}

Since $B^{\prime}$ admits a $3$-liec shown in Figure \ref{Fig_Bprime}, we
immediately obtain the following corollary.

\begin{corollary}
Let $G\not =B$ be a colorable cactus graph. Then {\normalsize $\chi_{%
%TCIMACRO{\TeXButton{TeX field}{\rm{irr}}}%
%BeginExpansion
\rm{irr}%
%EndExpansion
}^{\prime}(G)\leq3.$}
\end{corollary}

\section{Concluding remarks}

It was recently established by the present authors that the Local Irregularity
Conjecture does not hold in general, since there exists the so called bow-tie
graph $B$ which is colorable, but requires at least four colors for a locally
irregular edge coloring. The bow-tie graph $B$ is a cactus graph, and it was
previously established that only a certain family of cacti, besides odd length
paths and odd length cycles, is not colorable. This all indicated that the
cacti are quite relevant class of graphs for the Local Irregularity
Conjecture. Our previous investigation which led to the counterexample $B,$
made us believe that bow-tie graph $B$ might be the lonely counterexample in
the class of cacti. In this paper, we established that $B$ is indeed the only
counterexample in the class of cacti, i.e. that all other colorable cacti
require at most $3$ colors for the locally irregular coloring. We believe this
holds even in general, so we propose the following conjecture.

\begin{conjecture}
The bow-tie graph $B$ is the only colorable connected graph with $\chi_{%
%TCIMACRO{\TeXButton{TeX field}{\rm{irr}}}%
%BeginExpansion
\rm{irr}%
%EndExpansion
}^{\prime}(B)>3$.
\end{conjecture}

\bigskip\noindent\textbf{Acknowledgments.}~~Both authors acknowledge partial
support of the Slovenian research agency ARRS program\ P1-0383 and ARRS
project J1-1692. The first author also the support of Project
KK.01.1.1.02.0027, a project co-financed by the Croatian Government and the
European Union through the European Regional Development Fund - the
Competitiveness and Cohesion Operational Programme.

\newpage

\section{Appendix}

\begin{proof}
[Proof of Proposition \ref{Prop_primaryColoring}]We already stated that
alternative primary colorings of alternative berries are shown in Figure
\ref{Fig_alternativeBerries}. For a standard berry $G$ rooted at $u,$ let $v$
be a neighbor of $u$ and $w$ the other neighbor of $u$ if $d_{G}(u)=2.$ Also,
denote $T=G-u$ and notice that $T$ is a tree. Assume first that $G$ is a
unicyclic berry. If $T$ is not colorable, then $G\not =B_{1},B_{2}$ implies
that there exists a vertex $z\not =w$ in $T$ such that $T^{\prime}=T-vz$ is
one or two even length paths. Coloring $uv,uw$ and $vz$ by the color $c$ and
$T^{\prime}$ by colors $a$ and $b$ yields a standard primary coloring of $G,$
which is not a liec.

So, let us assume that $T$ is colorable. If $\chi_{%
%TCIMACRO{\TeXButton{TeX field}{\rm{irr}}}%
%BeginExpansion
\rm{irr}%
%EndExpansion
}^{\prime}(T)\leq2\emph{,}$ then coloring $uv$ and $uw$ by the color $c,$ and
$T$ by colors $a$ and $b,$ yields a standard primary of $G$ which is a liec.
Assume, therefore, that $\chi_{%
%TCIMACRO{\TeXButton{TeX field}{\rm{irr}}}%
%BeginExpansion
\rm{irr}%
%EndExpansion
}^{\prime}(T)=3.$ If $\Delta(T)=3$ and the only pair of vertices of degree $3$
in $T$ are $v$ and its acyclic neighbor, then there exists an edge $wz$ in $T$
such that $T^{\prime}=T-wz$ admits a $2$-liec $\phi_{a,b}^{\prime}.$ Coloring
edges $uv,$ $uw$ and $wz$ by $c,$ and $T^{\prime}$ by $\phi_{a,b}^{\prime}$
yields a standard primary coloring of $G$ which is not a liec. Otherwise,
$G\not =B_{3},B_{4},B_{5},B_{6}$ implies we can choose a rainbow root $z$ of
$T$ such that $z$ is neither from $\{v,w\}$, nor a docking vertex for
$A_{1}^{\mathrm{{op}}}$ in $G$. According to Observation
\ref{Observation_maxDeg} there exists a $3$-liec $\phi_{a,b,c}^{T}$ of $T$
with the rainbow root $z$ in which color $c$ is used in only one of the shrubs
of $T$ rooted at $z$. Notice that $v$ and $w$ belong to at most two shrubs of
$T,$ so we may assume $\phi_{a,b,c}^{T}$ uses the color $c$ in a third shrub
of $T$ not containing $v$ and $w.$ Thus, coloring $uv$ and $uw$ by the color
$c$ and $T$ by $\phi_{a,b,c}^{T}$ yields a standard primary coloring of $G$
which is also a liec of $G.$

It remains to consider a standard acyclic berry $G$ rooted at $u.$ If $T=G-u$
is not colorable, there exists an edge $vz$ in $T$ such that $T^{\prime}=T-vz$
is one or two even length paths. Coloring $T^{\prime}$ by $a$ and $b,$ and
$uv$ and $vz$ by $c$ yields the result. If $T$ admits a $2$-liec, then
coloring $T$ by $a$ and $b,$ and $uv$ by $c$ yields the result. Finally, if
$T$ requires $3$ colors for a liec, Observation \ref{Observation_maxDeg}
implies there exists a $3$-liec $\phi_{a,b,c}^{T}$ of $T$ such that the
rainbow root $z$ of $T$ is not a docking vertex for $A_{1}^{\mathrm{{op}}}.$
If $z\not =v,$ then we assume $\phi_{a,b,c}^{T}$ uses the color $c$ in a shrub
not containing $v,$ otherwise we assume $\phi_{a,b,c}^{T}$ uses the color $c$
in precisely one shrub so that the $c$-degree of the neighbor of $v$ is
$\geq3$. Thus, coloring $T$ by $\phi_{a,b,c}^{T}$ and $uv$ by $c$ yields the result.
\end{proof}

\begin{remark}
\label{Obs_standardNotLiec}A standard primary coloring of a standard unicyclic
berry constructed in the proof of Proposition \ref{Prop_primaryColoring} is
not a liec only in two following cases:

\begin{itemize}
\item[(i)] $T$ is colorable, $\chi_{%
%TCIMACRO{\TeXButton{TeX field}{\rm{irr}}}%
%BeginExpansion
\rm{irr}%
%EndExpansion
}^{\prime}(T)=3$, $\Delta(T)=3$ and the only vertices of degree $3$ in $T$ are
$v$ and its acyclic neighbor (resp. $w$ and its acyclic neighbor);

\item[(ii)] $T$ is not colorable.
\end{itemize}
\end{remark}

\medskip

\begin{proof}
[Proof of Proposition \ref{Prop_secondaryColoring1}]Notice that $d_{G}(u)=1$
implies $G$ is an acyclic berry. Let $G^{\prime}=G+ux$ for $x\not \in V(G).$
Since $G$ is not an even length path, $G^{\prime}$ admits a $3$-liec
$\phi_{a,b,c}^{\prime}.$ Since $x$ is a leaf in $G^{\prime}$ and
$d_{G^{\prime}}(u)=2,$ we conclude that $u$ must be monochromatic by
$\phi_{a,b,c}^{\prime},$ say $\phi_{a,b,c}^{\prime}(u)=\{c\}.$ Let $v$ be the
neighbor of $u$ in $G,$ then $d_{G^{\prime}}^{c}(u)=2$ implies $d_{G^{\prime}%
}^{c}(v)\not =2.$ So, the restriction of $\phi_{a,b,c}^{\prime}$ to $G$ is a
secondary coloring of $G.$
\end{proof}

\medskip

\begin{proof}
[Proof of Proposition \ref{Prop_secondaryColoring2}]It is sufficient to prove
that berries from Observation \ref{Obs_standardNotLiec} admit a secondary
coloring. We will consider separately both cases. Let $T^{\prime}=G-uv.$ In
case (i), we may assume $w$ and its acyclic neighbors are of degree $3$ in
$T.$ Hence, $w$ is the only vertex of degree $4$ in $T^{\prime},$ so
Observation \ref{Observation_maxDeg} implies $T^{\prime}$ admits a $2$-liec
$\phi_{a,b}^{T^{\prime}}$. Coloring $uv$ by $c$ and $T^{\prime}$ by
$\phi_{a,b}^{T^{\prime}}$ yields a secondary coloring of $G$ for which
$d_{G}^{c}(v)=1.$

In case (ii), the tree $T^{\prime}$ is either an even length path or
$T^{\prime}$ has precisely one vertex of degree $3,$ that is the vertex $w,$
and three paths hanging at $w.$ Either way, $T^{\prime}$ admits a $2$-liec
$\phi_{a,b}^{T^{\prime}}$. Then again, coloring $uv$ by $c$ and $T^{\prime}$
by $\phi_{a,b}^{T^{\prime}}$ yields a secondary coloring of $G$ for which
$d_{G}^{c}(v)=1.$
\end{proof}

\medskip

\begin{proof}
[Proof of Proposition \ref{Prop_grapes1}]This proposition is a direct
consequence of Proposition \ref{Prop_grapes2} and the fact that $B^{\prime}$
admits a $3$-liec which is established by Figure \ref{Fig_Bprime}.
\end{proof}

\medskip

Before we proceed, let us make the following observation.

\begin{remark}
\label{Obs_B7}In a locally irregular coloring of the berry $B_{7},$ the only
vertex of degree $3$ can be $1$- or $3$-chromatic (i.e. it cannot be $2$-chromatic).
\end{remark}

Propositions \ref{Prop_reduction1}, \ref{Prop_reduction2} and
\ref{Prop_reduction3} are a direct consequence of one more general proposition
which we will state and prove later, so now we proceed with Proposition
\ref{Prop_ColoringExtendableClosure1}.

\medskip

\begin{proof}
[Proof of Proposition \ref{Prop_ColoringExtendableClosure1}]For berries
colored by an alternative primary and a tertiary coloring, the respective
coloring of the closure is illustrated by Figures \ref{Fig_alternativeBerries}
and \ref{Fig_tertiaryCol}. Now, let $G$ be a berry colored by a standard
primary coloring $\phi_{a,b,c}$ as in the proof of Proposition
\ref{Prop_primaryColoring} and let $x\in V(G)\backslash\{u\}$ be a docking
vertex for an acyclic structure $A_{i}^{\mathrm{{op}}}$ where $i\in
\{1,\ldots,6\}$. If $i=1,$ then $x$ cannot belong to the cycle of $G,$ so $x$
is $3$-chromatic by $\phi_{a,b,c}$ only if $x$ is the rainbow root of $T.$ But
we always chose the rainbow root so that it is not a docking vertex for
$A_{1}^{\mathrm{{op}}},$ hence $x$ is $1$- or $2$-chromatic by $\phi_{a,b,c}.$
Observation \ref{Obs_B7} implies $x$ is $1$-chromatic by $\phi_{a,b,c},$ and
the additional two edges of the closure can be colored by any of the two
remaining colors, thus obtained standard primary coloring of the closure is
also a liec. This settles $A_{1}^{\mathrm{{op}}},$ and since acyclic berries
can only contain $A_{1}^{\mathrm{{op}}}$, this settles them too.

Assume that $G$ is a standard unicyclic berry and $i\not =1.$ To make clear
the notation, for all vertices within $A_{i}^{\mathrm{{op}}}$ we will use
labels from Figure \ref{Fig_singularGrapes}, only the label of the docking
vertex will be $x$ instead of $u$ in order to distinguish the root of $G$ and
the root of $A_{i}^{\mathrm{{op}}}.$

\medskip\noindent\textbf{Case 1: }$G$\emph{ is a unicyclic berry where
}$x\not \in \{v,w\}$\emph{ and }$x$\emph{ is not the rainbow root of }%
$T=G-u$\emph{.} This implies $x$ is $1$- or $2$-chromatic, say $\phi
_{a,b,c}(x)\subseteq\{a,b\}.$ If $x$ is a docking vertex for $A_{3}%
^{\mathrm{{op}}}$ or $A_{5}^{\mathrm{{op}}},$ all the additional edges of the
closure can be colored by the color $c.$

If $x$ is a docking vertex for $A_{2}^{\mathrm{{op}}},$ since $v_{1}$ and
$v_{2}$ must be monochromatic in $G,$ we have $d_{G}^{a}(x)\not =2$ and
$d_{G}^{b}(x)\not =2,$ so we may assume $d_{G}^{a}(x)\geq3$ and $\phi
_{a,b,c}(w_{1}v_{1})=a.$ Thus, coloring $uw_{1}$ by $c$ and changing color of
$w_{1}v_{1}$ also to $c$ yields the desired coloring of the closure.

If $x$ is a docking vertex for $A_{4}^{\mathrm{{op}}},$ by Observation
\ref{Obs_B7} we may assume $\phi_{a,b,c}(v_{2})=\{b\},$ which implies
$d_{G}^{b}(x)\not =3$ and (since $v_{1}$ must be monochromatic by
$\phi_{a,b,c}$) also $d_{G}^{b}(x)\not =2$. If $d_{G}^{b}(x)=1$ then the
additional edge $xw_{1}$ can be $b$-colored, otherwise if $d_{G}^{b}(x)=4$
then $xw_{1}$ can be $a$-colored provided that the color of $w_{1}v_{1}$ is
also changed from $b$ to $a.$

If $x$ is a docking vertex for $A_{6}^{\mathrm{{op}}},$ since $v_{1}$ must be
monochromatic in $G,$ we have $d_{G}^{a}(x)\not =2.$ We may assume $d_{G}%
^{a}(x)\not =1$, as otherwise we could swap the colors of the pendant path
containing $v_{1}$ and the berry $B_{7}$ containing $v_{2}.$ Now, $d_{G}%
^{a}(x)\geq3$ implies that the additional edge $xw_{1}$ can be $c$-colored,
provided that the color of $v_{1}w_{1}$ also be changed from $a$ to $c$.

Notice that the obtained coloring of $G^{\mathrm{{cl}}}$ in this case is also
a liec.

\medskip\noindent\textbf{Case 2: }$G$\emph{ is a unicyclic berry where }%
$x$\emph{ is the rainbow root of }$T=G-u$\emph{.} The choice of the rainbow
root in the proof of Proposition \ref{Prop_primaryColoring} implies
$x\not \in \{v,w\}.$ Notice that $x$ cannot be a docking vertex for
$A_{5}^{\mathrm{{op}}}$ or $A_{6}^{\mathrm{{op}}},$ since $x\not \in \{v,w\}$
would imply $d_{T}(x)\geq5,$ i.e. $T$ would admit a $2$-liec and would not
have a rainbow root. Also, $x$ cannot be a docking vertex for $A_{2}%
^{\mathrm{{op}}}=A_{3}^{\mathrm{{op}}},$ since then $d_{T}(x)=4$ and $x$ has
two neighbors of degree $2,$ so the $a$-sequence of $x$ cannot be $4,3,3,2$
which is the necessary sequence of a rainbow root of degree $4.$

Finally, if $x$ is the docking vertex for $A_{4}^{\mathrm{{op}}},$ then $x$
belongs to the cycle of $G$, so let $x_{1}$ and $x_{2}$ be the neighbors of
$x$ from the cycle of $G.$ Let $\phi_{a,b,c}^{T}$ be a $3$-liec of $T$ with
the rainbow root $x.$ Since $d_{G}(v_{1})=2,$ it follows that two of the edges
$x_{1}x,$ $x_{2}x$ and $v_{2}x$ must be colored by a same color by
$\phi_{a,b,c}^{T},$ say $x_{1}x$ and $x_{2}x$ are $a$-colored, $xv_{2}$ is
$b$-colored and $xv_{1}$ is $c$-colored. Then the additional edge $xw_{1}$ can
be $b$-colored. The obtained standard primary coloring of $G^{\mathrm{{cl}}}$
is a liec.

\medskip\noindent\textbf{Case 3: }$G$\emph{ is a unicyclic berry where }%
$x\in\{v,w\}$\emph{.} In this case $x$ is not the rainbow root of $T.$ If $x$
is a docking vertex for $A_{2}^{\mathrm{{op}}}=A_{3}^{\mathrm{{op}}},$ then
$d_{T}(x)=3,$ so Observation \ref{Obs_B7} implies that $x$ is monochromatic in
$T,$ say in color $a.$ Then in case of $A_{3}^{\mathrm{{op}}}$ the edges
$xw_{1}$ and $xw_{2}$ of $G^{\mathrm{{cl}}}$ can be $b$-colored, and in case
of $A_{2}^{\mathrm{{op}}}$ the edge $xw_{1}$ of $G^{\mathrm{{cl}}}$ can be
$b$-colored provided that the color of $w_{1}v_{1}$ is also changed from $a$
to $b.$ In both cases, the obtained primary coloring of $G^{\mathrm{{cl}}}$ is
a liec. If $x$ is a docking vertex for $A_{4}^{\mathrm{{op}}}$ or
$A_{6}^{\mathrm{{op}}},$ then the additional edge $xw_{1}$ of $G^{\mathrm{{cl}%
}}$ can be colored by $c.$ This yields a standard primary coloring of
$G^{\mathrm{{cl}}}$ which is not a liec.

Assume now that $x$ is a docking vertex for $A_{5}^{\mathrm{{op}}}.$ Since
$d_{G}(x)=5,$ we have $d_{T}(x)=4.$ Notice that $d_{T}^{a}(x)=d_{T}^{b}(x)=2$
would imply $xv_{1}$ and $xv_{2}$ are not locally irregular by $\phi_{a,b,c},$
a contradiction. So, we may assume $d_{T}^{a}(x)\geq3,$ which further implies
$\phi_{a,b,c}(xv_{1})=\phi_{a,b,c}(xv_{2})=a.$ Now, if $d_{T}^{a}(x)=4,$ then
the additional edges of $G^{\mathrm{{cl}}}$ can be colored by $b.$ Otherwise,
if $d_{T}^{a}(x)=3,$ then $\phi_{a,b,c}(xv_{3})=b,$ so we change color $b$
into $c$ for all edges of $B_{7}$ contained in $A_{5}$, and then all the
additional edges of $G^{\mathrm{{cl}}}$ can be colored by $b$. This yields a
standard primary coloring of $G^{\mathrm{{cl}}}$ which is not a liec.\medskip

It remains to consider the case when $G$ is a berry colored by a secondary
coloring. If $G$ is a unicyclic berry, then for $x\not =v$ (resp. $x=v$) the
proof is the same as in Case 1 (resp. Case 3) of this proof for a standard
primary coloring. If $G$ is an acyclic berry, then its secondary coloring is
obtained as the restriction to $G$ of any liec of $G^{\prime}=G+uy$ where
$y\not \in V(G).$ Since $G$ is acyclic, a docking vertex $x$ of $G$ can be a
docking vertex only for $A_{1}^{\mathrm{{op}}}.$ When coloring $G^{\prime},$
by Observation \ref{Observation_maxDeg} we can always choose the rainbow root
so that it is not a docking vertex for $A_{1}^{\mathrm{{op}}}.$ Thus,
Observation \ref{Obs_B7} implies $x$ is monochromatic by $\phi_{a,b,c},$ so
the additional two edges of $G^{\mathrm{{cl}}}$ can be colored by any of the
remaining two colors.

Finally, as for end-grapes which arise in $G^{\mathrm{{cl}}}$ they are
$A_{i}^{\mathrm{{op}}}$ with $i\in\{1,\ldots,6\},$ so their berries can be
$B_{1},$ $B_{7}$ or an even length path, and their edges must be locally
irregular by a primary (resp. secondary, tertiary) coloring of
$G^{\mathrm{{cl}}}.$ The restriction of a locally irregular coloring to
$B_{1}$ or an even length path can only be a primary coloring. As for $B_{7},$
we always chose rainbow root so that it is not a docking vertex for
$A_{1}^{\mathrm{{op}}},$ so Observation \ref{Obs_B7} implies $B_{7}$ is also
colored by a primary coloring.
\end{proof}

It will be useful to have a secondary coloring of a standard unicyclic berry
or its closure for which a primary coloring is not a liec, so let us observe
the following.

\begin{remark}
\label{Obs_standardClosureNotLiec}A primary coloring of a closure of a
standard unicyclic berry constructed in the proof of Proposition
\ref{Prop_ColoringExtendableClosure1} is not a liec in the case when closure
$G^{\mathrm{{cl}}}$ of $G$ which is obtained by closing $A_{i}^{\mathrm{{op}}%
}$ for $i\in\{4,5,6\}$ hanging at $v$ or $w$ (here, it is possible that
$A_{i}^{\mathrm{{op}}}$ for $i\in\{4,5,6\}$ hangs at one or both $v$ and $w,$
and if on both it does not have to be a same structure).
\end{remark}

The following lemma is the counterpart of Proposition
\ref{Prop_secondaryColoring2} for closures.

\begin{lemma}
\label{Lemma_secondaryClosure}Let $G$ be a standard berry rooted at $u$ with
$d_{G}(u)=2$ and $G^{\mathrm{{cl}}}$ a closure of $G$. If $G^{\mathrm{{cl}}}$
does not admit a primary coloring which is a liec, then $G^{\mathrm{{cl}}}$
admits a secondary coloring.
\end{lemma}

\begin{proof}
It is sufficient to prove the lemma for closures from Observation
\ref{Obs_standardClosureNotLiec}. We may assume that $G^{\mathrm{{cl}}}$ is
obtained by closing $A_{4}^{\mathrm{{op}}},$ $A_{5}^{\mathrm{{op}}}$ or
$A_{6}^{\mathrm{{op}}}$ hanging at $w$ (if such structure is hanging at $v$
and not at $w,$ we switch labels $v$ and $w$). Let $T^{\prime}=G-uv.$ If
$A_{5}^{\mathrm{{op}}}$ or $A_{6}^{\mathrm{{op}}}$ is hanging at $w$, then
$d_{T^{\prime}}(w)=5$ implies $T^{\prime}$ admits a $2$-liec $\phi
_{a,b}^{T^{\prime}}$. If $A_{4}^{\mathrm{{op}}}$ is hanging at $w,$ then
$d_{T^{\prime}}(w)=4$ and $d_{T^{\prime}}(u)=1,$ so the neighbors of $w$ in
$T^{\prime}$ cannot form the degree sequence $4,3,3,2,$ hence $T^{\prime}$
again admits a $2$-liec $\phi_{a,b}^{T^{\prime}}.$ Now, $\phi_{a,b}%
^{T^{\prime}}$ can be extended to a secondary coloring $\phi_{a,b,c}$ of $G$
by coloring the edge $uv$ by $c,$ and the coloring $\phi_{a,b,c}$ can be
extended to a secondary coloring of $G^{\mathrm{{cl}}}$ in the same way as in
the proof of Proposition \ref{Prop_ColoringExtendableClosure1}.

Notice that in the case when $A_{4}^{\mathrm{{op}}},$ $A_{5}^{\mathrm{{op}}}$
or $A_{6}^{\mathrm{{op}}}$ is hanging at $v$ too, it might happen that the
$c$-degree of $v$ in $G^{\mathrm{{cl}}}$ is $2,$ so we have to prove the
fourth property of a secondary coloring of $G^{\mathrm{{cl}}},$ i.e. that
$d_{G^{\mathrm{{cl}}}}^{c}(v)=2$ implies $d_{G^{\mathrm{{cl}}}}^{a}(w)\geq3.$
Since $u$ is a leaf in $T^{\prime}$ we have $d_{T^{\prime}}^{a}(w)\geq2.$ If
$A_{4}^{\mathrm{{op}}}$ is hanging at $w,$ since $v_{1}$ must be monochromatic
in $T^{\prime}$ we further have $d_{T^{\prime}}^{a}(w)\geq3.$ The additional
edge $ww_{1}$ of $G^{\mathrm{{cl}}}$ is colored by $c,$ so we conclude
$d_{G^{\mathrm{{cl}}}}^{a}(w)\geq3.$ If $A_{5}^{\mathrm{{op}}}$ is hanging at
$w,$ then $d_{T^{\prime}}^{a}(w)=2$ would imply $wv_{3}$ is the other
$a$-colored edge incident to $w.$ But then the additional edges $ww_{1}$ and
$ww_{2}$ of the closure can be $a$-colored, and thus $d_{G^{\mathrm{{cl}}}%
}^{a}(w)=4\geq3.$ Finally, if $A_{6}^{\mathrm{{op}}}$ is hanging at $w,$ then
$d_{T^{\prime}}^{a}(w)=2$ would imply $wv_{2}$ or $wv_{3}$ is locally regular,
a contradiction.
\end{proof}

\medskip

In order to prove Propositions \ref{Prop_reduction1}, \ref{Prop_reduction2}
and \ref{Prop_reduction3}, we need the definition of one specific end-grape
and several lemmas. An \emph{end-grape} $B_{2}^{\ast}$ is an end-grape
consisting of just one berry and that berry is a non-colorable berry on
triangle which has only one exit edge, as is illustrated by Figure
\ref{Fig_A7}. Notice that the pending paths hanging at $v$ and/or $w$ in
$B_{2}^{\ast}$ can also be of the length zero.

\begin{figure}[h]
\begin{center}
\includegraphics[scale=0.6]{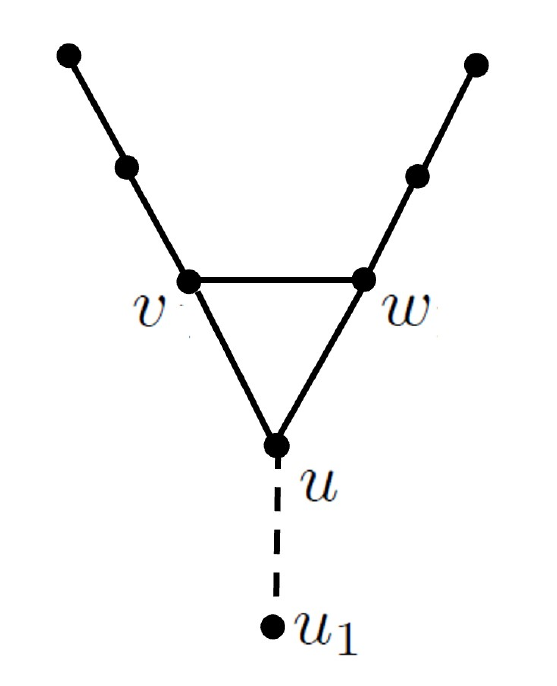}
\end{center}
\caption{End-grape $B_{2}^{\ast}.$}%
\label{Fig_A7}%
\end{figure}

\begin{lemma}
\label{Lemma_nonColorable}Let $G$ be a colorable cactus graph with
$\mathfrak{c}\geq2$ cycles. If $G$ contains an end-grape $G_{u}=B_{2}^{\ast}$
with a root component $G_{0}$, then $G_{0}^{\prime}$ is colorable and any
$k$-liec of $G_{0}^{\prime}$ with $k\geq3$ can be extended to a $k$-liec of
$G$ in which the only berry of $G_{u}$ is colored by a primary coloring.
\end{lemma}

\begin{proof}
Denote by $v$ and $w$ the neighbors of $u$ in $G_{u}.$ Since $G_{u}$ is a
single non-colorable triangular berry, it consists of the triangle $uvw$ and
possibly an even length path hanging at $v$ and/or $w$. Recall that
$G_{0}^{\prime}=G_{0}+uv$ and $u_{1}$ denotes the only neighbor of $u$ in
$G_{0}.$ If $G_{0}^{\prime}$ is non-colorable, then $G$ would be
non-colorable, a contradiction. Thus, $G_{0}^{\prime}$ is colorable and let
$\phi_{a,b,c}^{\prime0}$ be a $k$-liec of $G_{0}^{\prime}$ for $k\geq3.$ Since
$G_{u}$ has a single exit edge, we have $d_{G_{0}^{\prime}}(u)=2$ and
$d_{G_{0}^{\prime}}(v)=1.$ Thus, $vu$ and $uu_{1}$ must be colored by a same
color by $\phi_{a,b,c}^{\prime0},$ say color $c.$ We can now extend the
coloring $\phi_{a,b,c}^{\prime0}$ of $G_{0}^{\prime}$ to a coloring of $G$ by
coloring edges $uw$ and $wv$ by $a,$ and the remaining edges of $G_{u}$ by
$b.$ Thus we obtain a $k$-liec of $G$ in which the only berry of $G_{u}$ is
colored by an alternative primary coloring (since $G_{u}=B_{1}$ or
$G_{u}=B_{2},$ so $G_{u}$ is an alternative berry).
\end{proof}

Now we can start proving the reductions. In order to do so, we need to
introduce a primary coloring of an end-grape and its modification. First, let
us introduce a standard notation in $G_{u}$ and $G_{0}.$ Each berry of $G_{u}$
is denoted by $G_{i}$, where:

\begin{itemize}
\item $G_{i}$ is an alternative unicyclic berry for $i=1,\ldots,p;$

\item $G_{i}$ is a standard unicyclic berry for $i=p+1,\ldots,p+q;$

\item $G_{i}$ is an acyclic berry (standard or alternative) for
$i=p+q+1,\ldots,p+q+r.$
\end{itemize}

Let $v_{i}$ be a neighbor of $u$ in $G_{i},$ and $w_{i}$ the other neighbor of
$u$ if $d_{G_{i}}(u)=2.$ As for the root component $G_{0},$ recall that
$d_{G_{0}}(u)\leq2$ and the two neighbors of $u$ in $G_{0}$ are denoted by
$u_{1}$ and $u_{2}.$ Also, recall that $G_{0}^{\prime}=G_{0}+uv_{1}.$ The
degree of a vertex $v$ in $G$ (resp. $G_{u},$ $G_{i},$ $G_{0},$ $G_{0}%
^{\prime}$) will be denoted by $d(v)$ (resp. $d_{u}(v),$ $d_{i}(v),$
$d_{0}(v),$ $d_{0}^{\prime}(v)$).

Now, we define a \emph{primary coloring} of an end-grape $G_{u}$ by
\[
\phi_{a,b,c}^{\mathrm{{prim}}}=\sum_{i=1}^{p+q+r}\phi_{a,b,c}%
^{i,\mathrm{{prim}}},
\]
where $\phi_{a,b,c}^{i,\mathrm{{prim}}}$ denotes a primary coloring of the
berry $G_{i}.$ The color degree, say $a$-degree, of a vertex $v\in V(G_{u})$
by $\phi_{a,b,c}^{\mathrm{{prim}}}$ will be denoted by $d_{\mathrm{{prim}}%
}^{a}(v).$ Notice the following, if $d_{\mathrm{{prim}}}^{c}(u)\geq4$ then
every $c$-colored edge of $G_{u}$ incident to $u$ is locally irregular by
$\phi_{a,b,c}^{\mathrm{{prim}}}$, and if $d_{\mathrm{{prim}}}^{c}(u)\geq3$
then the same holds only if $G_{u}$ does not contain a berry $B_{7}$. As for
$a$-colored edges incident to $u,$ namely edges $uw_{i}$ in alternative
unicyclic berries, they are locally irregular by $\phi_{a,b,c}%
^{i,\mathrm{{prim}}},$ but they might not be locally irregular by
$\phi_{a,b,c}^{\mathrm{{prim}}}.$ To be more precise, some of the $a$-colored
edges $uw_{i}$ will not be locally irregular in two cases:

\begin{itemize}
\item $d_{\mathrm{{prim}}}^{a}(u)=2$ and there exists an alternative berry
$G_{i}$ in $G_{u}$ such that $d_{\mathrm{{prim}}}^{a}(w_{i})=2,$

\item $d_{\mathrm{{prim}}}^{a}(u)=4$ and there exists an alternative berry
$G_{i}$ in $G_{u}$ such that $d_{\mathrm{{prim}}}^{a}(w_{i})=4.$
\end{itemize}

In order to avoid this, we define a \emph{modified primary coloring}
$\phi_{a,b,c}^{u}$ of an end-grape $G_{u}$ in the following way:

\begin{itemize}
\item If every $a$-colored edge by $\phi_{a,b,c}^{\mathrm{{prim}}}$ is locally
irregular in $G_{u},$ we define $\phi_{a,b,c}^{u}=\phi_{a,b,c}^{\mathrm{{prim}%
}};$

\item if $d_{\mathrm{{prim}}}^{a}(u)=2$ and there exists an alternative berry
$G_{i}$ in $G_{u}$ with $d_{\mathrm{{prim}}}^{a}(w_{i})=2,$ then $\phi
_{a,b,c}^{u}$ is defined as the coloring obtained from $\phi_{a,b,c}%
^{\mathrm{{prim}}}$ by swapping colors $a$ and $c$ in every alternative
unicyclic berry with $d_{\mathrm{{prim}}}^{a}(w_{i})=2;$

\item if $d_{\mathrm{{prim}}}^{a}(u)=4$ and there exists an alternative berry
$G_{i}$ in $G_{u}$ with $d_{\mathrm{{prim}}}^{a}(w_{i})=4,$ then $\phi
_{a,b,c}^{u}$ is defined as the coloring obtained from $\phi_{a,b,c}%
^{\mathrm{{prim}}}$ by swapping colors $a$ and $b$ in precisely one
alternative unicyclic berry with $d_{\mathrm{{prim}}}^{a}(w_{i})=4.$
\end{itemize}

Notice that $d_{\mathrm{{prim}}}^{c}(u)=d_{u}^{c}(u)$. Also, $\phi_{a,b,c}%
^{u}$ is a liec of $G_{u}$ if $d_{u}^{c}(u)\geq4$ or $d_{u}^{c}(u)\geq3$ and
$B_{7}\not \in G_{u}$. Lastly, if $u$ is $3$-chromatic by $\phi_{a,b,c}^{u},$
then $d_{u}^{a}(u)=3$ and $d_{u}^{b}(u)=1.$

\medskip

\begin{proof}
[Proof of Proposition \ref{Prop_grapes2}]Denote by $\phi_{a,b,c}$ the desired
$3$-liec of $G.$ We will consider entire $G$ to be an end-grape of itself,
i.e. $G=G_{u}$ where $u$ is the vertex shared by all cycles in $G.$ Since
$\mathfrak{c}\geq2,$ it follows that $d_{\mathrm{{prim}}}^{c}(u)\geq2.$ If
$d_{\mathrm{{prim}}}^{c}(u)\geq4,$ then $\phi_{a,b,c}=\phi_{a,b,c}^{u}.$ So,
let us consider the remaining two cases, that is $d_{\mathrm{{prim}}}%
^{c}(u)=3$ and $d_{\mathrm{{prim}}}^{c}(u)=2.$

Assume first that $d_{\mathrm{{prim}}}^{c}(u)=3,$ and notice that
$\phi_{a,b,c}^{u}$ will be the desired $3$-liec of $G$ in all cases except
when $G$ contains a berry $B_{7}.$ If $G$ contains a berry $B_{7},$ notice
that $d_{\mathrm{{prim}}}^{c}(u)=3$ and $\mathfrak{c}\geq2$ imply $p=2$ and
$r=1.$ Then $G\not =B^{\prime}$ implies $G_{1}\not =B_{1},$ so $G_{1}$ admits
a tertiary coloring. Now, if $d_{\mathrm{{tert}}}^{c}(w_{1})=3$ we have
$\phi_{a,b,c}=\phi_{a,b,c}^{1,\mathrm{{tert}}}+\phi_{a,b,c}^{2,\mathrm{{prim}%
}}+\phi_{a,b,c}^{3,\mathrm{{prim}}}$, and if $d_{\mathrm{{tert}}}^{c}%
(w_{1})=4$ we have $\phi_{a,b,c}=\phi_{a,b,c}^{1,\mathrm{{tert}}}+\phi
_{a,b,c}^{2,\mathrm{{prim}}}+\phi_{a,c,b}^{3,\mathrm{{prim}}}.$

Assume now that $d_{\mathrm{{prim}}}^{c}(u)=2.$ Since $\mathfrak{c}\geq2$, the
only possibility is $p=2.$ If $d_{\mathrm{{prim}}}^{c}(v_{1}%
)=d_{\mathrm{{prim}}}^{c}(v_{2})=1$ we have $\phi_{a,b,c}=\phi_{a,b,c}%
^{1,\mathrm{{prim}}}+\phi_{b,a,c}^{2,\mathrm{{prim}}},$ else if
$d_{\mathrm{{prim}}}^{c}(v_{1})=1$ and $d_{\mathrm{{prim}}}^{c}(v_{2})=2,$
since $d_{\mathrm{{prim}}}^{c}(v_{2})=2$ implies $d_{\mathrm{{prim}}}%
^{a}(w_{2})=4,$ we have $\phi_{a,b,c}=\phi_{a,b,c}^{1,\mathrm{{prim}}}%
+\phi_{c,a,b}^{2,\mathrm{{prim}}}.$ Finally if $d_{\mathrm{{prim}}}^{c}%
(v_{1})=d_{\mathrm{{prim}}}^{c}(v_{2})=2,$ the fourth property of an
alternative primary coloring implies $d_{\mathrm{{prim}}}^{a}(w_{1}%
)=d_{\mathrm{{prim}}}^{a}(w_{2})=4,$ so we have $\phi_{a,b,c}=\phi
_{a,b,c}^{1,\mathrm{{prim}}}+\phi_{a,c,b}^{2,\mathrm{{prim}}}.$
\end{proof}

\medskip

Now, Propositions \ref{Prop_reduction1}, \ref{Prop_reduction2} and
\ref{Prop_reduction3} can be unified into one more general proposition, but
let us first prove the following four lemmas which will be necessary to prove
it. In all the following lemmas, for a coloring $\phi_{a,b,c}^{0}$ of a root
component $G_{0}$ we assume $\phi_{a,b,c}^{0}(u_{1}u)=a$ and $\phi_{a,b,c}%
^{0}(u_{2}u)\in\{a,b\}.$

\begin{lemma}
\label{Lemma_reduction_deg4}Let $G$ be a cactus graph with $\mathfrak{c}\geq2$
cycles which is not a grape and does not contain an end-grape $B_{2}^{\ast}$.
If $G$ contains an end-grape $G_{u}$ with $d_{\mathrm{{prim}}}^{c}(u)\geq4,$
then every $k$-liec of $G_{0}$ for $k\geq3$ can be extended to a $k$-liec of
$G$ so that every berry of $G_{u}$ is colored by a primary, secondary or
tertiary coloring.
\end{lemma}

\begin{proof}
Denote by $\phi_{a,b,c}$ the desired $k$-liec of $G.$ If $G_{0}$ is
non-colorable, since $G$ is not a grape, we conclude $G_{0}\in\mathfrak{T}$.
Thus, $G$ would contain an end-grape $B_{2}^{\ast}$ within $G_{0},$ a
contradiction. So, we may assume that $G_{0}$ is colorable.

Now, the first possibility is that $G_{u}$ has a single exit edge. If
$d_{0}^{a}(u_{1})\not =d_{u}^{c}(u)+1$ we have $\phi_{a,b,c}=\phi_{c,b,a}%
^{0}+\phi_{a,b,c}^{u},$ and if $d_{0}^{a}(u_{1})=d_{u}^{c}(u)+1$ then
$d_{u}^{b}(u)<d_{u}^{c}(u)$ implies $\phi_{a,b,c}=\phi_{b,a,c}^{0}%
+\phi_{a,b,c}^{u}$.

The second possibility is that $G_{u}$ has a double monochromatic exit edge.
If $d_{0}^{a}(u_{1})\not =d_{u}^{c}(u)+2$ and $d_{0}^{a}(u_{2})\not =d_{u}%
^{c}(u)+2,$ we have $\phi_{a,b,c}=\phi_{c,b,a}^{0}+\phi_{a,b,c}^{u}.$
Otherwise, if $d_{0}^{a}(u_{1})=d_{u}^{c}(u)+2$ and $d_{0}^{a}(u_{2}%
)\not =d_{u}^{b}(u)+2,$ from $d_{u}^{b}(u)<d_{u}^{c}(u)$ we have $\phi
_{a,b,c}=\phi_{b,a,c}^{0}+\phi_{a,b,c}^{u}.$ Lastly, if $d_{0}^{a}%
(u_{1})=d_{u}^{c}(u)+2$ and $d_{0}^{a}(u_{2})=d_{u}^{b}(u)+2,$ notice that $u$
must be $3$-chromatic by $\phi_{a,b,c}^{u}$ as otherwise $d_{0}^{a}%
(u_{2})=d_{u}^{b}(u)+2=2$ would imply $u_{2}u$ is locally regular by
$\phi_{a,b,c}^{0},$ a contradiction. So, $d_{u}^{b}(u)<d_{u}^{a}(u)<d_{u}%
^{c}(u)$ implies $\phi_{a,b,c}=\phi_{a,b,c}^{0}+\phi_{a,b,c}^{u}.$

The third possibility is that $G_{u}$ has a double bichromatic exit edge. If
$d_{0}^{a}(u_{1})\not =d_{u}^{c}(u)+1$ and $d_{0}^{b}(u_{2})\not =d_{u}%
^{b}(u)+1,$ we have $\phi_{a,b,c}=\phi_{c,b,a}^{0}+\phi_{a,b,c}^{u}.$ If
$d_{0}^{a}(u_{1})\not =d_{u}^{c}(u)+1$ and $d_{0}^{b}(u_{2})=d_{u}^{b}(u)+1,$
notice that $u$ must be $3$-chromatic by $\phi_{a,b,c}^{u},$ as otherwise
$d_{0}^{b}(u_{2})=d_{u}^{b}(u)+1=1$ would imply $u_{2}u$ is locally regular by
$\phi_{a,b,c}^{0},$ so $\phi_{a,b,c}=\phi_{c,a,b}^{0}+\phi_{a,b,c}%
^{\mathrm{{prim}}}.$ Finally, if $d_{0}^{a}(u_{1})=d_{u}^{c}(u)+1$ and
$d_{0}^{b}(u_{2})=d_{u}^{c}(u)+1,$ we distinguish the cases when $u$ is $1$-,
$2$- and $3$-chromatic by $\phi_{a,b,c}^{u}.$ If $u$ is $1$-chromatic by
$\phi_{a,b,c}^{u},$ then $\phi_{a,b,c}=\phi_{a,b,c}^{0}+\phi_{a,b,c}^{u}.$
Otherwise, if $u$ is $2$-chromatic by $\phi_{a,b,c}^{u},$ let $\phi
_{a,b,c}^{\prime u}$ be a coloring of $G_{u}$ obtained from $\phi_{a,b,c}^{u}$
either by swapping colors $a$ and $b$ in one unicyclic berry $G_{i}$ with
$d_{u}^{a}(w_{i})=4$ or by swapping colors $a$ and $c$ and then $a$ and $b$ in
one unicyclic berry $G_{i}$ with $d_{u}^{a}(w_{i})=2$, then $\phi_{a,b,c}%
=\phi_{a,b,c}^{0}+\phi_{a,b,c}^{\prime u}.$ Lastly, if $u$ is $3$-chromatic by
$\phi_{a,b,c}^{u},$ let $\phi_{a,b,c}^{\prime u}$ be a coloring of $G_{u}$
obtained from $\phi_{a,b,c}^{u}$ by swapping colors $a$ and $b$ in one
alternative unicyclic berry $G_{i}$ with $\phi_{a,b,c}^{u}(uw_{i})=a,$ then
$\phi_{a,b,c}=\phi_{a,b,c}^{0}+\phi_{a,b,c}^{\prime u}.$
\end{proof}

\begin{lemma}
\label{Lemma_reduction_deg3}Let $G$ be a cactus graph with $\mathfrak{c}\geq2$
cycles which is not a grape and does not contain an end-grape $B_{2}^{\ast}$.
If $G$ contains an end-grape $G_{u}\not =A_{5},A_{6}$ with $d_{\mathrm{{prim}%
}}^{c}(u)=3,$ then every $k$-liec of $G_{0}$ for $k\geq3$, can be extended to
a $k$-liec of $G$ so that every berry of $G_{u}$ is colored by a primary,
secondary and tertiary coloring.
\end{lemma}

\begin{proof}
Denote by $\phi_{a,b,c}$ the desired $k$-liec of $G.$ First, notice that
$d_{\mathrm{{prim}}}^{c}(u)=d_{u}^{c}(u)=3$ implies $u$ is $1$- or
$2$-chromatic by $\phi_{a,b,c}^{u},$ i.e. $\phi_{a,b,c}^{u}(u)\subseteq
\{a,c\}.$ Similarly as in previous lemma, we may assume $G_{0}$ is colorable.
Now, we distinguish the following cases.

\medskip\noindent\textbf{Case 1: }$p=1$ and $q=1.$ In this case $d_{u}%
^{a}(u)=1.$ Also, $r=0$ implies that $G_{u}$ does not contain $B_{7},$ so
$\phi_{a,b,c}^{u}$ is a liec of $G_{u}.$ If $u$ is $1$-chromatic by
$\phi_{a,b,c}^{0}$ we have $\phi_{a,b,c}=\phi_{b,a,c}^{0}+\phi_{a,b,c}^{u}.$
If $u$ is $2$-chromatic by $\phi_{a,b,c}^{0},$ then if $d_{0}^{a}(u_{1}%
)\not =4$ we have $\phi_{a,b,c}=\phi_{c,b,a}^{0}+\phi_{a,b,c}^{u},$ else if
$d_{0}^{a}(u_{1})=d_{0}^{b}(u_{2})=4$ then in case $d_{u}^{a}(w_{1})=4$ we
have $\phi_{a,b,c}=\phi_{a,b,c}^{0}+\phi_{a,b,c}^{u}$ and in case $d_{u}%
^{a}(w_{1})=2$ the fourth property of an alternative primary coloring implies
$d_{u}^{c}(v_{1})=1,$ so we have $\phi_{a,b,c}=\phi_{a,b,c}^{0}+\phi
_{c,b,a}^{1,\mathrm{{prim}}}+\phi_{a,b,c}^{2,\mathrm{{prim}}}$.

\medskip\noindent\textbf{Case 2: }$q=1$ and $r=1.$ In this case $\phi
_{a,b,c}^{u}(u)=\{c\}$, i.e. $d_{u}^{a}(u)=0.$ If $G_{2}\not =B_{7},$ then
$\phi_{a,b,c}^{u}$ is a liec of $G_{u},$ so $\phi_{a,b,c}=\phi_{a,b,c}%
^{0}+\phi_{a,b,c}^{u}.$ Hence, we may assume $G_{2}=B_{7}.$ The first
possibility is that $G_{u}$ has a single exit edge. If $d_{0}^{a}(u_{1}%
)\not =4$ we have $\phi_{a,b,c}=\phi_{c,b,a}^{0}+\phi_{a,b,c}^{u}$, and if
$d_{0}^{a}(u_{1})=4$ we have $\phi_{a,b,c}=\phi_{c,b,a}^{0}+\phi
_{a,b,c}^{1,\mathrm{{prim}}}+\phi_{a,c,b}^{2,\mathrm{{prim}}}.$ The second
possibility is that $G_{u}$ has double monochromatic exit edge. Here, if
$d_{0}^{a}(u_{1})\not =5$ and $d_{0}^{a}(u_{2})\not =5$ we have $\phi
_{a,b,c}=\phi_{c,b,a}^{0}+\phi_{a,b,c}^{u}.$ Otherwise, if $d_{0}^{a}%
(u_{1})=5$ and $d_{0}^{a}(u_{2})\not =4$ we have $\phi_{a,b,c}=\phi
_{c,b,a}^{0}+\phi_{a,b,c}^{1,\mathrm{{prim}}}+\phi_{c,b,a}^{2,\mathrm{{prim}}%
}.$ Finally, if $d_{0}^{a}(u_{1})=5$ and $d_{0}^{a}(u_{2})=4,$ then if
$\phi_{a,b,c}^{1,\mathrm{{prim}}}$ is a liec of $G_{1}$ (resp.
$G^{\mathrm{{cl}}}$) we have $\phi_{a,b,c}=\phi_{a,b,c}^{0}+\phi
_{c,b,a}^{1,\mathrm{{prim}}}+\phi_{a,c,b}^{2,\mathrm{{prim}}}$. If $G$ (resp.
$G^{\mathrm{{cl}}}$) does not admit a primary liec, then according to
Proposition \ref{Prop_secondaryColoring2} (resp. Lemma
\ref{Lemma_secondaryClosure}) it admits a secondary coloring and we have
$\phi_{a,b,c}=\phi_{c,b,a}^{0}+\phi_{a,b,c}^{1,\mathrm{{sec}}}+\phi
_{a,c,b}^{2,\mathrm{{prim}}}.$ The third and the last possibility is that
$G_{u}$ has double bichromatic exit edge. If $d_{0}^{a}(u_{1})\not =4$ we have
$\phi_{a,b,c}=\phi_{c,b,a}^{0}+\phi_{a,b,c}^{u}$, and if $d_{0}^{a}%
(u_{1})=d_{0}^{b}(u_{2})=4$ we have $\phi_{a,b,c}=\phi_{c,b,a}^{0}%
+\phi_{a,b,c}^{1,\mathrm{{prim}}}+\phi_{c,b,a}^{2,\mathrm{{prim}}}.$

\medskip\noindent\textbf{Case 3: }$p=3.$ In this case $d_{u}^{a}(u)=3.$ Since
$B_{7}\not \in G_{u},$ then $\phi_{a,b,c}^{u}$ is a liec of $G_{u}.$ If $u$ is
monochromatic by $\phi_{a,b,c}^{0}$ we have $\phi_{a,b,c}=\phi_{b,a,c}%
^{0}+\phi_{a,b,c}^{u},$ so we may assume that $u$ is bichromatic by
$\phi_{a,b,c}^{0}.$ If $d_{0}^{a}(u_{1})\not =4$ we have $\phi_{a,b,c}%
=\phi_{c,b,a}^{0}+\phi_{a,b,c}^{u},$ and if $d_{0}^{a}(u_{1})=d_{0}^{b}%
(u_{2})=4$ we distinguish cases when $d_{\mathrm{{prim}}}^{a}(w_{1})=2$ and
$d_{\mathrm{{prim}}}^{a}(w_{1})=4.$ If $d_{\mathrm{{prim}}}^{a}(w_{1})=2,$
then the fourth property of an alternative primary coloring implies
$d_{\mathrm{{prim}}}^{c}(v_{1})=1,$ so we have $\phi_{a,b,c}=\phi_{a,b,c}%
^{0}+\phi_{c,a,b}^{1,\mathrm{{prim}}}+\sum_{i=2}^{3}\phi_{a,b,c}%
^{i,\mathrm{{prim}}}$, and if $d_{\mathrm{{prim}}}^{a}(w_{1})=4$ we have
$\phi_{a,b,c}=\phi_{a,b,c}^{0}+\phi_{b,a,c}^{1,\mathrm{{prim}}}+\sum_{i=2}%
^{3}\phi_{a,b,c}^{i,\mathrm{{prim}}}$.

\medskip\noindent\textbf{Case 4: }$p=2$ and $r=1.$ Notice that in this case
$d_{u}^{a}(u)=2.$ Assume first $G_{3}\not =B_{7},$ so $\phi_{a,b,c}^{u}$ is a
liec of $G_{u}$. If $u$ is monochromatic by $\phi_{a,b,c}^{0}$ we have
$\phi_{a,b,c}=\phi_{b,a,c}^{0}+\phi_{a,b,c}^{u},$ so we may assume $u$ is
bichromatic by $\phi_{a,b,c}^{0}.$ If $d_{0}^{a}(u_{1})\not =4$ we have
$\phi_{a,b,c}=\phi_{c,b,a}^{0}+\phi_{a,b,c}^{u}$, and if $d_{0}^{a}%
(u_{1})=d_{0}^{b}(u_{2})=4$ we distinguish cases when $d_{\mathrm{{prim}}}%
^{c}(v_{1})=1$ and $d_{\mathrm{{prim}}}^{c}(v_{1})=2.$ If $d_{\mathrm{{prim}}%
}^{c}(v_{1})=1$ we have $\phi_{a,b,c}=\phi_{a,c,b}^{0}+\phi_{b,c,a}%
^{1,\mathrm{{prim}}}+\sum_{i=2}^{3}\phi_{a,b,c}^{i,\mathrm{{prim}}},$ and if
$d_{\mathrm{{prim}}}^{c}(v_{1})=2$, then $d_{\mathrm{{prim}}}^{a}(w_{1})=4,$
so we have $\phi_{a,b,c}=\phi_{a,c,b}^{0}+\phi_{a,c,b}^{1,\mathrm{{prim}}%
}+\sum_{i=2}^{3}\phi_{a,b,c}^{i,\mathrm{{prim}}}.$

Assume now that $G_{3}=B_{7}.$ The first possibility is that $G_{u}$ has a
single exit edge. Here, if $d_{0}^{a}(u_{1})\not =4$ we have $\phi
_{a,b,c}=\phi_{c,a,b}^{0}+\phi_{a,b,c}^{u},$ and if $d_{0}^{a}(u_{1})=4$ we
have $\phi_{a,b,c}=\phi_{c,a,b}^{0}+\phi_{a,b,c}^{\prime u},$ where
$\phi_{a,b,c}^{\prime u}$ is obtained from $\phi_{a,b,c}^{u}$ by swapping
colors $b$ and $c$ in $G_{3}.$ The second possibility is that $G_{u}$ has
double monochromatic exit edge. If $d_{0}^{a}(u_{1})\not =5$ and $d_{0}%
^{a}(u_{2})\not =5$ we have $\phi_{a,b,c}=\phi_{c,a,b}^{0}+\phi_{a,b,c}^{u},$
and if $d_{0}^{a}(u_{1})=5$ and $d_{0}^{a}(u_{2})\not =4$ we have
$\phi_{a,b,c}=\phi_{c,a,b}^{0}+\phi_{a,b,c}^{\prime u},$ where $\phi
_{a,b,c}^{\prime u}$ is obtained from $\phi_{a,b,c}^{u}$ by swapping colors
$b$ and $c$ in $G_{3}.$ Finally, if $d_{0}^{a}(u_{1})=5$ and $d_{0}^{a}%
(u_{2})=4,$ then $G_{u}\not =A_{5}$ implies $G_{1}\not =B_{1},$ so we have
$\phi_{a,b,c}=\phi_{c,a,b}^{0}+\phi_{a,b,c}^{1,\mathrm{{tert}}}+\sum_{i=2}%
^{3}\phi_{a,b,c}^{i,\mathrm{{prim}}}$.

The third and the last possibility is that $G_{u}$ has double bichromatic exit
edge. Here, if $d_{0}^{a}(u_{1})\not =4$ we have $\phi_{a,b,c}=\phi
_{c,b,a}^{0}+\phi_{a,b,c}^{u},$ and if $d_{0}^{a}(u_{1})=d_{0}^{b}(u_{2})=4$
we have $\phi_{a,b,c}=\phi_{c,b,a}^{0}+\phi_{a,b,c}^{\prime u},$ where
$\phi_{a,b,c}^{\prime u}$ is obtained from $\phi_{a,b,c}^{u}$ by swapping
colors $b$ and $c$ in $G_{3}.$

\medskip\noindent\textbf{Case 5: }$p=1$ and $r=2.$ In this case $d_{u}%
^{a}(u)=1.$ Assume first that $G_{2}\not =B_{7}$ and $G_{3}\not =B_{7},$ so
$\phi_{a,b,c}^{u}$ is a liec of $G_{u}.$ If $u$ is monochromatic by
$\phi_{a,b,c}^{0},$ then $\phi_{a,b,c}=\phi_{b,a,c}^{0}+\phi_{a,b,c}^{u},$ so
we may assume that $u$ is bichromatic by $\phi_{a,b,c}^{0}.$ If $d_{0}%
^{a}(u_{1})\not =4$ we have $\phi_{a,b,c}=\phi_{c,b,a}^{0}+\phi_{a,b,c}^{u}$,
and if $d_{0}^{a}(u_{1})=d_{0}^{b}(u_{2})=4$ we distinguish cases when
$G_{1}=B_{1}$ and $G_{1}\not =B_{1}.$ If $G_{1}=B_{1}$ we have $\phi
_{a,b,c}=\phi_{c,b,a}^{0}+\phi_{a,c,b}^{1,\mathrm{{prim}}}+\sum_{i=2}^{3}%
\phi_{a,b,c}^{i,\mathrm{{prim}}},$ and if $G_{1}\not =B_{1}$ we have
$\phi_{a,b,c}=\phi_{c,b,a}^{0}+\phi_{a,b,c}^{1,\mathrm{{tert}}}+\sum_{i=2}%
^{3}\phi_{a,b,c}^{i,\mathrm{{prim}}}.$

Assume now that $G_{2}\not =B_{7}$ and $G_{3}=B_{7},$ so $B_{7}\in G_{u}$
implies $\phi_{a,b,c}^{u}$ is not a liec. The first possibility is that
$G_{u}$ has a single exit edge. If $d_{0}^{a}(u_{1})\not =4$ we have
$\phi_{a,b,c}=\phi_{c,b,a}^{0}+\phi_{a,b,c}^{u}$, and if $d_{0}^{a}(u_{1})=4$
we have $\phi_{a,b,c}=\phi_{c,b,a}^{0}+\phi_{a,b,c}^{\prime u},$ where
$\phi_{a,b,c}^{\prime u}$ is obtained from $\phi_{a,b,c}^{u}$ by swapping
colors $b$ and $c$ in $G_{3}.$ The next possibility is that $G_{u}$ has double
monochromatic exit edge. If $d_{0}^{a}(u_{1})\not =5$ and $d_{0}^{a}%
(u_{2})\not =5,$ we have $\phi_{a,b,c}=\phi_{c,b,a}^{0}+\phi_{a,b,c}^{u}.$
Otherwise, if $d_{0}^{a}(u_{1})=5$ and $d_{0}^{a}(u_{2})\not =4,$ we have
$\phi_{a,b,c}=\phi_{c,b,a}^{0}+\phi_{a,b,c}^{\prime u}$ where $\phi
_{a,b,c}^{\prime u}$ is obtained from $\phi_{a,b,c}^{u}$ by swapping colors
$b$ and $c$ in $G_{3}.$ Finally, if $d_{0}^{a}(u_{1})=5$ and $d_{0}^{a}%
(u_{2})=4,$ then we distinguish cases whether $G_{1}$ is $B_{1}$ or not. If
$G_{1}=B_{1}$ then $\phi_{a,b,c}=\phi_{c,b,a}^{0}+\phi_{a,c,b}%
^{1,\mathrm{{prim}}}+\phi_{a,b,c}^{2,\mathrm{{prim}}}+\phi_{a,c,b}%
^{3,\mathrm{{prim}}},$ and if $G_{1}\not =B_{1}$ then $\phi_{a,b,c}%
=\phi_{c,b,a}^{0}+\phi_{a,b,c}^{1,\mathrm{{tert}}}+\sum_{i=2}^{3}\phi
_{a,b,c}^{i,\mathrm{{prim}}}$. The last possibility is that $G_{u}$ has double
bichromatic exit edge. If $d_{0}^{a}(u_{1})\not =4$ we have $\phi_{a,b,c}%
=\phi_{c,b,a}^{0}+\phi_{a,b,c}^{u},$ and if $d_{0}^{a}(u_{1})=d_{0}^{b}%
(u_{2})=4$ we have $\phi_{a,b,c}=\phi_{c,b,a}^{0}+\phi_{a,b,c}^{\prime u}$
where $\phi_{a,b,c}^{\prime u}$ is obtained from $\phi_{a,b,c}^{u}$ by
swapping colors $b$ and $c$ in $G_{3}.$

Assume lastly that $G_{2}=B_{7}$ and $G_{3}=B_{7},$ so again $\phi_{a,b,c}%
^{u}$ is not a liec of $G_{u}.$ The first possibility is that $G_{u}$ has the
single exit edge, in which case if $d_{0}^{a}(u_{1})\not =4$ we have
$\phi_{a,b,c}=\phi_{c,b,a}^{0}+\phi_{a,b,c}^{u}$, and if $d_{0}^{a}(u_{1})=4$
we distinguish cases when $G_{1}=B_{1}$ and $G_{1}\not =B_{1}.$ If
$G_{1}=B_{1}$ we have $\phi_{a,b,c}=\phi_{c,b,a}^{0}+\phi_{a,b,c}%
^{1,\mathrm{{prim}}}+\sum_{i=2}^{3}\phi_{a,c,b}^{i,\mathrm{{prim}}},$ and if
$G_{1}\not =B_{1}$ we have $\phi_{a,b,c}=\phi_{a,b,c}^{0}+\phi_{a,b,c}%
^{1,\mathrm{{tert}}}+\sum_{i=2}^{3}\phi_{a,b,c}^{i,\mathrm{{prim}}}.$ The next
possibility is that $G_{u}$ has double monochromatic exit edge. Here, if
$d_{0}^{a}(u_{1})\not =5$ and $d_{0}^{a}(u_{2})\not =5$ we have $\phi
_{a,b,c}=\phi_{c,b,a}^{0}+\phi_{a,b,c}^{u},$ if $d_{0}^{a}(u_{1})=5$ and
$d_{0}^{a}(u_{2})\not =4$ we have $\phi_{a,b,c}=\phi_{c,b,a}^{0}+\sum
_{i=1}^{2}\phi_{a,b,c}^{i,\mathrm{{prim}}}+\phi_{a,c,b}^{3,\mathrm{{prim}}},$
and if $d_{0}^{a}(u_{1})=5$ and $d_{0}^{a}(u_{2})=4$ we have $\phi
_{a,b,c}=\phi_{c,b,a}^{0}+\phi_{a,b,c}^{1,\mathrm{{prim}}}+\sum_{i=2}^{3}%
\phi_{a,c,b}^{i,\mathrm{{prim}}}.$ The last possibility is that $G_{u}$ has
double bichromatic exit edge. Here, if $d_{0}^{a}(u_{1})\not =4$ we have
$\phi_{a,b,c}=\phi_{c,b,a}^{0}+\phi_{a,b,c}^{u},$ and if $d_{0}^{a}%
(u_{1})=d_{0}^{b}(u_{2})=4,$ since $G_{u}\not =A_{6}$ implies $G_{1}%
\not =B_{1}$ we have $\phi_{a,b,c}=\phi_{c,b,a}^{0}+\phi_{a,b,c}%
^{1,\mathrm{{tert}}}+\sum_{i=2}^{3}\phi_{a,b,c}^{i,\mathrm{{prim}}}$.
\end{proof}

\begin{lemma}
\label{Lemma_reduction_deg2}Let $G$ be a cactus graph with $\mathfrak{c}\geq2$
cycles which is not a grape and does not contain an end-grape $B_{2}^{\ast}$.
If $G$ contains an end-grape $G_{u}\not =A_{i}$ for $i=1,2,3$ with
$d_{\mathrm{{prim}}}^{c}(u)=2,$ then there exists $\tilde{G}_{0}\in
\{G_{0},G_{0}^{\prime}\}$ such that every $k$-liec of $\tilde{G}_{0}$ for
$k\geq3$ can be extended to a $k$-liec of $G$ so that every berry of $G_{u}$
is colored by a primary, secondary and tertiary coloring.
\end{lemma}

\begin{proof}
Let $\phi_{a,b,c}$ denote the desired $k$-liec of $G.$ As in previous two
lemmas we may assume $G_{0}$ is colorable.

\medskip\noindent\textbf{Case 1: }$q=1.$ If $G_{1}$ (resp. $G_{1}%
^{\mathrm{{cl}}}$) admits a primary coloring which is also a liec, then we
have $\phi_{a,b,c}=\phi_{a,b,c}^{0}+\phi_{a,b,c}^{1,\mathrm{{prim}}}$. So, let
us assume $G_{1}$ (resp. $G_{1}^{\mathrm{{cl}}}$) does not admit a primary
coloring which is also a liec, then Proposition \ref{Prop_secondaryColoring2}
(resp. Lemma \ref{Lemma_secondaryClosure}) implies it admits a secondary
coloring $\phi_{a,b,c}^{1,\mathrm{{sec}}}.$ Now, we distinguish the cases when
$d_{\mathrm{{sec}}}^{c}(v_{1})=1$ and $d_{\mathrm{{sec}}}^{c}(v_{1})=2.$

Assume first that $d_{\mathrm{{sec}}}^{c}(v_{1})=1$ and let $G_{0}^{\prime
}=G_{0}+uv_{1}$. If $G_{0}^{\prime}$ is non-colorable, then the existence of
the leaf $v_{1}$ in $G_{0}^{\prime}$ implies $G_{0}^{\prime}\in\mathfrak{T},$
so $G$ would contain $B_{2}^{\ast}$ within $G_{0}^{\prime},$ a contradiction.
If $G_{0}^{\prime}$ is colorable, let $\phi_{a,b,c}^{\prime0}$ be a $k$-liec
of $G_{0}^{\prime}.$ We claim that $u$ is $1$- or $2$-chromatic by
$\phi_{a,b,c}^{\prime0}.$ To see this, notice that either $d_{G_{0}^{\prime}%
}(u)=2$ or $uv_{1}$ is a single pendant edge on a cycle of $G_{0}^{\prime}.$
Now, if $d_{G_{0}^{\prime}}(u)=2$, the claim is obvious. On the other hand, if
$uv_{1}$ is a pendant edge on a cycle of $G_{0}^{\prime},$ then the color of
$uv_{1}$ must be the same by $\phi_{a,b,c}^{\prime0}$ as the color of at least
one more edge of $G_{0}^{\prime}$ incident to $u$ which establishes the claim.
So, we may assume $\phi_{a,b,c}^{\prime0}(uv_{1})=a$ and $\phi_{a,b,c}%
^{\prime0}(u)\subseteq\{a,b\}.$ Let $\phi_{a,b,c}^{0}$ be the restriction of
$\phi_{a,b,c}^{\prime0}$ to $G_{0},$ then we have $\phi_{a,b,c}=\phi
_{c,b,a}^{0}+\phi_{a,b,c}^{1,\mathrm{{sec}}}.$

Assume now that $d_{\mathrm{{sec}}}^{c}(v_{1})=2,$ which implies $\phi
_{a,b,c}^{1,\mathrm{{sec}}}$ is a liec and $d_{\mathrm{{sec}}}^{a}(w_{1}%
)\geq3$. If $u$ is monochromatic by $\phi_{a,b,c}^{0}$ we have $\phi
_{a,b,c}=\phi_{b,a,c}^{0}+\phi_{a,b,c}^{1,\mathrm{{sec}}}$, so we may assume
$u$ is bichromatic by $\phi_{a,b,c}^{0}.$ Now, if $d_{0}^{a}(u_{1})\not =2$ we
have $\phi_{a,b,c}=\phi_{a,b,c}^{0}+\phi_{a,b,c}^{1,\mathrm{{sec}}},$ and if
$d_{0}^{a}(u_{1})=d_{0}^{b}(u_{2})=2$ we have $\phi_{a,b,c}=\phi_{c,b,a}%
^{0}+\phi_{a,b,c}^{1,\mathrm{{prim}}}$.

\medskip\noindent\textbf{Case 2: }$p=2.$ Since $G_{u}\not =A_{i}$ for $i=1,3,$
we may assume that $G_{1}\not =B_{1}.$ The first possibility is that $G_{u}$
has single exit edge. If $d_{0}^{a}(u_{1})\not =3,$ then $\phi_{a,b,c}%
=\phi_{c,b,a}^{0}+\phi_{a,b,c}^{u}.$ Otherwise, if $d_{0}^{a}(u_{1})=3,$ then
we distinguish the cases when $G_{2}\not =B_{1}$ and $G_{2}=B_{1}.$ When
$G_{2}\not =B_{1}$ we have $\phi_{a,b,c}=\phi_{c,b,a}^{0}+\phi_{a,b,c}%
^{1,\mathrm{{tert}}}+\phi_{a,b,c}^{2,\mathrm{{tert}}}$, and when $G_{2}=B_{1}$
if $d_{\mathrm{{tert}}}^{c}(w_{1})=3$ we have $\phi_{a,b,c}=\phi_{c,b,a}%
^{0}+\phi_{a,b,c}^{1,\mathrm{{tert}}}+\phi_{a,b,c}^{2,\mathrm{{prim}}},$ and
if $d_{\mathrm{{tert}}}^{c}(w_{1})=4$ we have $\phi_{a,b,c}=\phi_{b,a,c}%
^{0}+\phi_{a,b,c}^{1,\mathrm{{tert}}}+\phi_{a,b,c}^{2,\mathrm{{prim}}}$.

The next possibility is that $G_{u}$ has a double monochromatic exit edge. If
$d_{0}^{a}(u_{1})\not =4$ and $d_{0}^{a}(u_{2})\not =4,$ then $\phi
_{a,b,c}=\phi_{c,b,a}^{0}+\phi_{a,b,c}^{u}.$ Otherwise, if $d_{0}^{a}%
(u_{1})=4$ and $d_{0}^{a}(u_{2})\not =5,$ then $\phi_{a,b,c}=\phi_{c,b,a}%
^{0}+\phi_{a,b,c}^{1,\mathrm{{tert}}}+\phi_{a,b,c}^{2,\mathrm{{prim}}}.$
Finally, if $d_{0}^{a}(u_{1})=4$ and $d_{0}^{a}(u_{2})=5,$ then we distinguish
three cases with regard to $d_{\mathrm{{prim}}}^{c}(v_{1})$ and
$d_{\mathrm{{prim}}}^{c}(v_{2}).$ First, if $d_{\mathrm{{prim}}}^{c}%
(v_{1})=d_{\mathrm{{prim}}}^{c}(v_{2})=1,$ we have $\phi_{a,b,c}=\phi
_{c,b,a}^{0}+\phi_{c,a,b}^{1,\mathrm{{prim}}}+\phi_{a,c,b}^{2,\mathrm{{prim}}%
}.$ Next, if $d_{\mathrm{{prim}}}^{c}(v_{1})=d_{\mathrm{{prim}}}^{c}%
(v_{2})=2,$ this implies $d_{\mathrm{{prim}}}^{a}(w_{1})=d_{\mathrm{{prim}}%
}^{a}(w_{2})=4,$ so we have $\phi_{a,b,c}=\phi_{c,b,a}^{0}+\phi_{a,b,c}%
^{1,\mathrm{{prim}}}+\phi_{a,c,b}^{2,\mathrm{{prim}}}.$ Lastly, if
$d_{\mathrm{{prim}}}^{c}(v_{1})=2$ and $d_{\mathrm{{prim}}}^{c}(v_{2})=1,$
this implies $d_{\mathrm{{prim}}}^{a}(w_{1})=4,$ so we have $\phi_{a,b,c}%
=\phi_{c,b,a}^{0}+\phi_{a,b,c}^{1,\mathrm{{prim}}}+\phi_{b,c,a}%
^{2,\mathrm{{prim}}}.$

The last possibility is that $G_{u}$ has double bichromatic exit edge. If
$d_{0}^{a}(u_{1})\not =3,$ then $\phi_{a,b,c}=\phi_{c,b,a}^{0}+\phi
_{a,b,c}^{u}.$ Otherwise, if $d_{0}^{a}(u_{1})=d_{0}^{b}(u_{2})=3,$ then in
case when $G_{2}\not =B_{1}$ we have $\phi_{a,b,c}=\phi_{c,b,a}^{0}%
+\phi_{a,b,c}^{1,\mathrm{{tert}}}+\phi_{a,b,c}^{2,\mathrm{{tert}}}$, and in
case $G_{2}=B_{1}$ we distinguish the situation when $d_{\mathrm{{tert}}}%
^{c}(w_{1})=3$ and $d_{\mathrm{{tert}}}^{c}(w_{1})=4.$ If $d_{\mathrm{{tert}}%
}^{c}(w_{1})=3$ we have $\phi_{a,b,c}=\phi_{c,b,a}^{0}+\phi_{a,b,c}%
^{1,\mathrm{{tert}}}+\phi_{a,b,c}^{2,\mathrm{{prim}}},$ and if
$d_{\mathrm{{tert}}}^{c}(w_{1})=4$ we have $\phi_{a,b,c}=\phi_{a,b,c}^{0}%
+\phi_{a,b,c}^{1,\mathrm{{tert}}}+\phi_{c,b,a}^{2,\mathrm{{prim}}}.$

\medskip\noindent\textbf{Case 3: }$p=1$ and $r=1.$ Assume first that
$G_{1}=B_{1}.$ Then $G_{u}\not =A_{3},A_{4}$ implies that $G_{2}\not =%
P_{2k},B_{7},$ where $P_{2k}$ denotes any even length path. Since $G_{2}%
\not =P_{2k},$ the graph $G_{2}+uv_{1}$ is colorable, which together with
$G_{2}\not =B_{7}$ implies $G_{2}$ admits a secondary coloring. Notice that
$\phi_{a,b,c}^{1,\mathrm{{prim}}}+\phi_{a,b,c}^{2,\mathrm{{sec}}}$ is a liec
of $G_{u}$ by which $u$ is $2$-chromatic in colors $a$ and $c.$ Therefore, if
$u$ is $1$-chromatic by $\phi_{a,b,c}^{0}$ we have $\phi_{a,b,c}=\phi
_{b,a,c}^{0}+\phi_{a,b,c}^{1,\mathrm{{prim}}}+\phi_{a,b,c}^{2,\mathrm{{sec}}%
}.$ So, let us assume $u$ is $2$-chromatic by $\phi_{a,b,c}^{0}.$ If
$d_{0}^{a}(u_{1})\not =3$ we have $\phi_{a,b,c}=\phi_{c,b,a}^{0}+\phi
_{a,b,c}^{u},$ and if $d_{0}^{a}(u_{1})=d_{0}^{b}(u_{2})=3$ we have
$\phi_{a,b,c}=\phi_{c,b,a}^{0}+\phi_{a,b,c}^{1,\mathrm{{prim}}}+\phi
_{a,c,b}^{2,\mathrm{{sec}}}.$

Assume now that $G_{2}=P_{2k}.$ Then $G_{u}\not =A_{3}$ implies $G_{1}%
\not =B_{1}.$ Assume first that $G_{u}$ has a single exit edge. If $d_{0}%
^{a}(u_{1})\not =3$ we have $\phi_{a,b,c}=\phi_{c,b,a}^{0}+\phi_{a,b,c}^{u},$
and if $d_{0}^{a}(u_{1})=3$ we distinguish cases when $d_{\mathrm{{tert}}}%
^{c}(w_{1})=3$ and $d_{\mathrm{{tert}}}^{c}(w_{1})=4.$ If $d_{\mathrm{{tert}}%
}^{c}(w_{1})=3$ we have $\phi_{a,b,c}=\phi_{c,b,a}^{0}+\phi_{a,b,c}%
^{1,\mathrm{{tert}}}+\phi_{a,b,c}^{2,\mathrm{{prim}}},$ and if
$d_{\mathrm{{tert}}}^{c}(w_{1})=4$ we have $\phi_{a,b,c}=\phi_{a,b,c}^{0}%
+\phi_{a,b,c}^{1,\mathrm{{tert}}}+\phi_{a,b,c}^{2,\mathrm{{prim}}}.$ Assume
next that $G_{u}$ has a double monochromatic exit edge. If $d_{0}^{a}%
(u_{1})\not =4$ and $d_{0}^{a}(u_{2})\not =4$ we have $\phi_{a,b,c}%
=\phi_{c,b,a}^{0}+\phi_{a,b,c}^{u},$ if $d_{0}^{a}(u_{1})=4$ and $d_{0}%
^{a}(u_{2})\not =3$ we have $\phi_{a,b,c}=\phi_{c,b,a}^{0}+\phi_{a,b,c}%
^{1,\mathrm{{prim}}}+\phi_{a,c,b}^{2,\mathrm{{prim}}},$ and if $d_{0}%
^{a}(u_{1})=4$ and $d_{0}^{a}(u_{2})=3$ we have $\phi_{a,b,c}=\phi_{c,b,a}%
^{0}+\phi_{a,b,c}^{1,\mathrm{{tert}}}+\phi_{a,b,c}^{2,\mathrm{{prim}}}.$
Assume lastly that $G_{u}$ has double bichromatic exit edge. If $d_{0}%
^{a}(u_{1})\not =3$ we have $\phi_{a,b,c}=\phi_{c,b,a}^{0}+\phi_{a,b,c}^{u}.$
If $d_{0}^{a}(u_{1})=d_{0}^{b}(u_{2})=3$ we distinguish cases when
$d_{\mathrm{{tert}}}^{c}(w_{1})=3$ and $d_{\mathrm{{tert}}}^{c}(w_{1})=4.$ If
$d_{\mathrm{{tert}}}^{c}(w_{1})=3$ we have $\phi_{a,b,c}=\phi_{c,b,a}^{0}%
+\phi_{a,b,c}^{1,\mathrm{{tert}}}+\phi_{a,b,c}^{2,\mathrm{{prim}}},$ and if
$d_{\mathrm{{tert}}}^{c}(w_{1})=4$ we have $\phi_{a,b,c}=\phi_{a,b,c}^{0}%
+\phi_{a,b,c}^{1,\mathrm{{tert}}}+\phi_{a,b,c}^{2,\mathrm{{prim}}}.$

Assume next that $G_{2}=B_{7}.$ Then $G_{u}\not =A_{4}$ implies $G_{1}%
\not =B_{1}.$ Assume first that $G_{u}$ has a single exit edge. If
$d_{\mathrm{{prim}}}^{c}(v_{1})=1$ we have $\phi_{a,b,c}=\phi_{b,a,c}^{0}%
+\phi_{a,b,c}^{u},$ and if $d_{\mathrm{{prim}}}^{c}(v_{1})=2$ then the fourth
property of alternative primary coloring implies $d_{\mathrm{{prim}}}%
^{a}(w_{1})=4,$ so we have $\phi_{a,b,c}=\phi_{b,a,c}^{0}+\phi_{a,b,c}%
^{1,\mathrm{{prim}}}+\phi_{c,b,a}^{2,\mathrm{{prim}}}.$ Assume next that
$G_{u}$ has a double monochromatic exit edge. If $d_{0}^{a}(u_{1})\not =4$ and
$d_{0}^{a}(u_{2})\not =4$ we have $\phi_{a,b,c}=\phi_{c,b,a}^{0}+\phi
_{a,b,c}^{u},$ if $d_{0}^{a}(u_{1})=4$ and $d_{0}^{a}(u_{2})\not =3$ we have
$\phi_{a,b,c}=\phi_{c,b,a}^{0}+\phi_{a,b,c}^{1,\mathrm{{prim}}}+\phi
_{a,c,b}^{2,\mathrm{{prim}}},$ and if $d_{0}^{a}(u_{1})=4$ and $d_{0}%
^{a}(u_{2})=3$ we have $\phi_{a,b,c}=\phi_{c,b,a}^{0}+\phi_{a,b,c}%
^{1,\mathrm{{tert}}}+\phi_{a,b,c}^{2,\mathrm{{prim}}}.$ Assume next that
$G_{u}$ has double bichromatic exit edge. We distinguish cases when
$d_{\mathrm{{tert}}}^{c}(w_{1})=3$ and $d_{\mathrm{{tert}}}^{c}(w_{1})=4.$ In
case when $d_{\mathrm{{tert}}}^{c}(w_{1})=3,$ if $d_{0}^{a}(u_{1})\not =4$ we
have $\phi_{a,b,c}=\phi_{c,b,a}^{0}+\phi_{a,b,c}^{1,\mathrm{{tert}}}%
+\phi_{a,b,c}^{2,\mathrm{{prim}}},$ and if $d_{0}^{a}(u_{1})=d_{0}^{b}%
(u_{2})=4$ we have $\phi_{a,b,c}=\phi_{c,b,a}^{0}+\phi_{a,b,c}%
^{1,\mathrm{{prim}}}+\phi_{a,c,b}^{2,\mathrm{{prim}}}$ (this must be a liec
since $d_{\mathrm{{tert}}}^{c}(w_{1})=3$ implies $G_{1}=B_{2}$ or $B_{3}$
which further implies $d_{\mathrm{{prim}}}^{c}(v_{1})=1$). In case when
$d_{\mathrm{{tert}}}^{c}(w_{1})=4,$ if $d_{0}^{a}(u_{1})\not =3$ we have
$\phi_{a,b,c}=\phi_{c,b,a}^{0}+\phi_{a,b,c}^{1,\mathrm{{tert}}}+\phi
_{c,b,a}^{2,\mathrm{{prim}}},$ and if $d_{0}^{a}(u_{1})=d_{0}^{b}(u_{2})=3$ we
distinguish cases when $d_{\mathrm{{prim}}}^{c}(v_{1})=1$ and
$d_{\mathrm{{prim}}}^{c}(v_{1})=2.$ If $d_{\mathrm{{prim}}}^{c}(v_{1})=1$ we
have $\phi_{a,b,c}=\phi_{c,b,a}^{0}+\phi_{a,b,c}^{1,\mathrm{{prim}}}%
+\phi_{a,c,b}^{2,\mathrm{{prim}}},$ and if $d_{\mathrm{{prim}}}^{c}(v_{1})=2$
then $d_{\mathrm{{prim}}}^{a}(w_{1})=4$ so we have $\phi_{a,b,c}=\phi
_{a,b,c}^{0}+\phi_{a,b,c}^{1,\mathrm{{prim}}}+\phi_{a,c,b}^{2,\mathrm{{prim}}%
}.$

Assume lastly that $G_{1}\not =B_{1}$ and $G_{2}\not =P_{2k},B_{7}.$ Since
$G_{1}\not =B_{1},$ the berry $G_{1}$ admits a tertiary coloring. If
$d_{\mathrm{{tert}}}^{c}(w_{1})=4$ we have $\phi_{a,b,c}=\phi_{a,b,c}^{0}%
+\phi_{a,b,c}^{1,\mathrm{{tert}}}+\phi_{a,b,c}^{2,\mathrm{{prim}}},$ so let us
assume $d_{\mathrm{{tert}}}^{c}(w_{1})=3.$ The first possibility is that
$G_{u}$ has a single exit edge, in which case if $d_{0}^{a}(u_{1})\not =4$ we
have $\phi_{a,b,c}=\phi_{c,b,a}^{0}+\phi_{a,b,c}^{1,\mathrm{{tert}}}%
+\phi_{a,b,c}^{2,\mathrm{{prim}}},$ and if $d_{0}^{a}(u_{1})=4$ we have
$\phi_{a,b,c}=\phi_{c,b,a}^{0}+\phi_{a,b,c}^{1,\mathrm{{prim}}}+\phi
_{a,b,c}^{2,\mathrm{{prim}}}.$ The next possibility is that $G_{u}$ has double
monochromatic exit edge. Here, if $d_{0}^{a}(u_{1})\not =5$ and $d_{0}%
^{a}(u_{2})\not =5$ we have $\phi_{a,b,c}=\phi_{c,b,a}^{0}+\phi_{a,b,c}%
^{1,\mathrm{{tert}}}+\phi_{a,b,c}^{2,\mathrm{{prim}}},$ if $d_{0}^{a}%
(u_{1})=5$ and $d_{0}^{a}(u_{2})\not =4$ we have $\phi_{a,b,c}=\phi
_{c,b,a}^{0}+\phi_{a,b,c}^{1,\mathrm{{prim}}}+\phi_{a,b,c}^{2,\mathrm{{prim}}%
},$ and if $d_{0}^{a}(u_{1})=5$ and $d_{0}^{a}(u_{2})=4$, then in case
$d_{\mathrm{{prim}}}^{c}(v_{1})=1$ we have $\phi_{a,b,c}=\phi_{c,b,a}^{0}%
+\phi_{a,c,b}^{1,\mathrm{{prim}}}+\phi_{a,c,b}^{2,\mathrm{{sec}}}$ and in case
$d_{\mathrm{{prim}}}^{c}(v_{1})=2$ we have $\phi_{a,b,c}=\phi_{b,a,c}^{0}%
+\phi_{a,b,c}^{1,\mathrm{{prim}}}+\phi_{a,c,b}^{2,\mathrm{{prim}}}$. The last
possibility is that $G_{u}$ has double bichromatic exit edge. If $d_{0}%
^{a}(u_{1})\not =3$ we have $\phi_{a,b,c}=\phi_{c,b,a}^{0}+\phi_{a,b,c}^{u},$
and if $d_{0}^{a}(u_{1})=d_{0}^{b}(u_{2})=3$ then in case $d_{\mathrm{{tert}}%
}^{c}(w_{1})=3$ we have $\phi_{a,b,c}=\phi_{c,b,a}^{0}+\phi_{a,b,c}%
^{1,\mathrm{{tert}}}+\phi_{a,b,c}^{2,\mathrm{{prim}}},$ and in case
$d_{\mathrm{{tert}}}^{c}(w_{1})=4$ we have $\phi_{a,b,c}=\phi_{a,b,c}^{0}%
+\phi_{a,b,c}^{1,\mathrm{{tert}}}+\phi_{a,b,c}^{2,\mathrm{{prim}}}.$
\end{proof}

\begin{lemma}
\label{Lemma_reduction_deg1}Let $G$ be a cactus graph with $\mathfrak{c}\geq2$
cycles which is not a grape and does not contain an end-grape $B_{2}^{\ast}$.
If $G$ contains an end-grape $G_{u}$ with $d_{\mathrm{{prim}}}^{c}(u)=1,$ then
there exists $\tilde{G}_{0}\in\{G_{0},G_{0}^{\prime}\}$ such that every
$k$-liec of $\tilde{G}_{0}$ for $k\geq3$ can be extended to a $k$-liec of $G$
so that every berry of $G_{u}$ is colored by a primary, secondary and tertiary coloring.
\end{lemma}

\begin{proof}
The only possibility is $p=1.$ We distinguish the following two cases.

\medskip\noindent\textbf{Case 1: }$d_{\mathrm{{prim}}}^{c}(v_{1})=1.$ If
$G_{0}^{\prime}=G+uv_{1}$ is non-colorable, since $v_{1}$ is a leaf in
$G_{0}^{\prime}$ we conclude $G_{0}^{\prime}\in\mathfrak{T}$, but then $G$
contains $B_{2}^{\ast}$ within $G_{0}^{\prime},$ a contradiction. So, we may
assume $G_{0}^{\prime}$ is colorable. Let $\phi_{a,b,c}^{\prime0}$ be a
$k$-liec of $G_{0}^{\prime}.$ If $u$ is $3$-chromatic by $\phi_{a,b,c}%
^{\prime0},$ then $d_{G_{0}^{\prime}}(u)\leq3$ implies $uv_{1}$ is locally
regular by $\phi_{a,b,c}^{\prime0},$ a contradiction. So, we may assume
$\phi_{a,b,c}^{\prime0}(uv_{1})=a$ and $\phi_{a,b,c}^{\prime0}(u)\subseteq
\{a,b\}.$ Let $\phi_{a,b,c}^{0}$ be the restriction of $\phi_{a,b,c}^{\prime
0}$ to $G_{0},$ then $\phi_{a,b,c}=\phi_{c,b,a}^{0}+\phi_{a,b,c}%
^{1,\mathrm{{prim}}}.$

\medskip\noindent\textbf{Case 2: }$d_{\mathrm{{prim}}}^{c}(v_{1})=2.$ Notice
that $d_{\mathrm{{prim}}}^{c}(v_{1})=2$ implies $\phi_{a,b,c}%
^{1,\mathrm{{prim}}}$ is a liec. Thus, if $u$ is $1$-chromatic by
$\phi_{a,b,c}^{0}$, we have $\phi_{a,b,c}=\phi_{b,a,c}^{0}+\phi_{a,b,c}%
^{1,\mathrm{{prim}}}.$ So, let us assume that $u$ is $2$-chromatic by
$\phi_{a,b,c}^{0}$. If $d_{0}^{a}(u_{1})\not =2,$ notice that by the
definition of primary coloring $d_{\mathrm{{prim}}}^{c}(v_{1})=2$ implies
$d_{\mathrm{{prim}}}^{a}(w_{1})=4,$ so we have $\phi_{a,b,c}=\phi_{a,b,c}%
^{0}+\phi_{a,b,c}^{1,\mathrm{{prim}}}.$ Otherwise, if $d_{0}^{a}(u_{1}%
)=d_{0}^{b}(u_{2})=2,$ notice that $d_{\mathrm{{prim}}}^{c}(v_{1})=2$ implies
$\phi_{a,b,c}^{1,\mathrm{{prim}}}$ is the coloring of the closure of $B_{4}$
or $B_{5},$ but then $d_{\mathrm{{tert}}}^{c}(w_{1})=4$, so we have
$\phi_{a,b,c}=\phi_{c,b,a}^{0}+\phi_{a,b,c}^{1,\mathrm{{tert}}}$.
\end{proof}

Now, we can state the following proposition which is a direct consequence of
Lemmas \ref{Lemma_nonColorable}--\ref{Lemma_reduction_deg1}.

\begin{proposition}
\label{Prop_reductionFinal}Let $G$ be a cactus graph with $\mathfrak{c}\geq2$
cycles which is not a grape. If $G$ contains an end-grape $G_{u}\not =A_{i}$
for $i=1,\ldots,6,$ then there exists $\tilde{G}_{0}\in\{G_{0},G_{0}^{\prime
}\}$ such that every $k$-liec of $\tilde{G}_{0},$ for $k\geq3,$ can be
extended to a $k$-liec of $G$ so that every berry of $G_{u}$ is colored by a
primary, secondary or tertiary coloring.
\end{proposition}

Propositions \ref{Prop_reduction2} and \ref{Prop_reduction3} are a direct
consequence of Proposition \ref{Prop_reductionFinal}. As for Proposition
\ref{Prop_reduction1}, it remains to consider the situation when $G$ contains
only end-grapes $A_{4},$ $A_{5}$ or $A_{6}.$ Here, the claim is that a
coloring of $\tilde{G}_{0}\in\{G_{0},G_{0}^{\prime}\}$ can be extended to $G$,
but the coloring of the berries from end-grapes $A_{4},$ $A_{5}$ or $A_{6}$
does not have to be primary, secondary or tertiary coloring. This is easily
established, namely let $B_{7}$ be a berry contained in $A_{i}$ for
$i\in\{4,5,6\}$ and $v$ the neighbor of $u$ in $B_{7}.$ Graph $G^{\prime}$ is
obtained from $G$ by removing all edges of $B_{7}$ distinct from $uv,$ thus an
end-grape of $G^{\prime}$ with removed edges is no longer $A_{i}$ for
$i\in\{4,5,6\}.$ Proposition \ref{Prop_reductionFinal} implies the coloring of
$G^{\prime},$ where we can assume $uv$ is colored by $c$. Then the removed
edges of $G,$ which form two pending paths of $B_{7}$ hanging at $v$, can be
colored one by $a$ and the other by $b$. This yields a $3$-liec of $G,$ but
the berry $B_{7}$ is not colored by a primary, secondary or tertiary coloring.
That settles Proposition \ref{Prop_reductionFinal}. Before we proceed with the
proof of the main theorem, we need the following lemma.

\begin{lemma}
\label{Lemma_G0tree}Proposition \ref{Prop_reductionFinal} holds even when
$G_{0}$ of $G$ is replaced by any connected graph $G_{0}^{\ast}$ with
$d_{G_{0}^{\ast}}(u)\leq2$.
\end{lemma}

\begin{proof}
If $G_{0}$ is colorable, then in the proof of Proposition
\ref{Prop_reductionFinal} we relied only on the color degrees of $u_{1}$ and
$u_{2}$ and not the structure of $G_{0},$ and if $G_{0}$ is non-colorable we
relied on $B_{2}^{\ast}$ being contained in $G_{0}.$ So, the claim is implied
by Proposition \ref{Prop_reductionFinal}, except when $G_{0}^{\ast}$ is a
path. Let $G_{0}^{\ast}$ be a path, we may assume the length of $G_{0}^{\ast}$
is at most two. Let $\phi_{c}^{0}$ be a $1$-edge coloring of $G_{0}^{\ast}$.
If $d_{u}^{c}(u)\geq3$ or $d_{u}^{c}(u)=2$ and $B_{7}\not \in G_{u},$ then
$\phi_{a,b,c}=\phi_{c}^{0}+\phi_{a,b,c}^{u}.$ If $d_{u}^{c}(u)=2$ and
$B_{7}\in G_{u},$ then $G_{u}$ contains only one unicyclic berry which is
alternative, so if $d_{u}^{c}(v_{1})=1$ then $\phi_{a,b,c}=\phi_{c}^{0}%
+\phi_{a,b,c}^{1,\mathrm{{prim}}}+\phi_{a,c,b}^{2,\mathrm{{prim}}},$ and if
$d_{u}^{c}(v_{1})=2$ then $\phi_{a,b,c}=\phi_{b}^{0}+\phi_{a,b,c}%
^{1,\mathrm{{prim}}}+\phi_{a,c,b}^{2,\mathrm{{prim}}}.$ Lastly, if $d_{u}%
^{c}(u)=1,$ then again $G_{u}$ contains only one alternative unicyclic berry,
so if $d_{u}^{c}(v_{1})=1$ then $\phi_{a,b,c}=\phi_{c}^{0}+\phi_{a,b,c}%
^{1,\mathrm{{prim}}}$ and if $d_{u}^{c}(v_{1})=2$ then $\phi_{a,b,c}=\phi
_{a}^{0}+\phi_{a,b,c}^{1,\mathrm{{prim}}}$ for odd length $G_{0}^{\ast}$ and
$\phi_{a,b,c}=\phi_{b}^{0}+\phi_{a,b,c}^{1,\mathrm{{prim}}}$ for even length
$G_{0}^{\ast}$.
\end{proof}

\medskip

\begin{proof}
[Proof of Theorem \ref{Tm_main}]The proof is by induction on the number of
cycles $\mathfrak{c}$. Assume first that $\mathfrak{c}=2.$ If $G$ is a grape,
the claim follows from Proposition \ref{Prop_grapes2}, so let us assume that
$G$ is not a grape. Hence, $G$ contains two end-grapes, one $G_{u}$ rooted at
$u$ and the other $G_{\bar{u}}$ rooted at $\bar{u}.$ Let us denote $\bar
{G}_{0}=G-(E(G_{u})\cup E(G_{\bar{u}}))$ and notice that $\bar{G}_{0}$ is a
tree connecting end-grapes $G_{u}$ and $G_{\bar{u}}$ of $G$ (as in our
notation we ignore isolated vertices). Let $G_{0}=\bar{G}_{0}\cup G_{\bar{u}}$
and notice that $G_{0}$ is a unicyclic graph, also $G_{u}$ is an end-grape of
$G$ with the root component $G_{0}.$ Applying Proposition
\ref{Prop_reductionFinal} to an end-grape $G_{u}$ of $G,$ and then Lemma
\ref{Lemma_G0tree} to an end-grape $G_{\bar{u}}$ of $\tilde{G}_{0}\in
\{G_{0},G_{0}^{\prime}\},$ yields the result.

Assume now that the claim holds for all cacti with less than $\mathfrak{c}>2$
cycles, and let $G$ be a cactus graph with $\mathfrak{c}$ cycles. Again, if
$G$ is a grape, the claim holds by Proposition \ref{Prop_grapes2}, so we may
assume $G$ is not a grape. If $G$ contains an end-grape $G_{u}$ distinct from
$A_{i},$ for $i=1,\ldots,6,$ then induction hypothesis implies there exists a
$3$-liec of $\tilde{G}_{0}\in\{G_{0},G_{0}^{\prime}\}$ such that every berry
of every end-grape is colored by a primary, secondary or tertiary coloring.
Proposition \ref{Prop_reductionFinal} further implies that the $3$-liec of
$\tilde{G}_{0}$ can be extended to a $3$-liec of $G$ such that every berry of
$G_{u}$ is colored by a primary, secondary or tertiary coloring, and the claim
is established.

So, let us assume that every end-grape $G_{u}$ of $G$ is a copy of $A_{i}$ for
$i=1,\ldots,6.$ Let ${G}^{\mathrm{{op}}}$ be the graph obtained from $G$ by
opening all triangles of all end-grapes of $G.$ Notice that ${G}%
^{\mathrm{{op}}}$ may be a tree, a unicyclic graph or a cactus graph with at
least two cycles.

\medskip\noindent\textbf{Case 1: }$G^{\mathrm{{op}}}$\emph{ is a tree.} In
this case ${G}^{\mathrm{{op}}}$ admits a $3$-liec ${\phi_{a,b,c}%
^{\mathrm{{op}}}}.$ Also, since ${G}^{\mathrm{{op}}}$ is a tree and $A_{1}$ is
only end-grape among $A_{i}$ with $i\in\{1,\ldots,6\}$ whose root does not
belong to a cycle, every end-grape of $G$ must be $A_{1}$. Now, $G\not =B$
implies that the rainbow root of $\phi_{a,b,c}^{\mathrm{{op}}}$ (if $\chi_{%
%TCIMACRO{\TeXButton{TeX field}{\rm{irr}}}%
%BeginExpansion
\rm{irr}%
%EndExpansion
}^{\prime}(G^{\mathrm{{op}}})=3$) can be chosen so that it is not a docking
vertex for $A_{1}^{\mathrm{{op}}}$ in $G^{\mathrm{{op}}},$ thus Observation
\ref{Obs_B7} implies that every docking vertex of $G^{\mathrm{{op}}}$ is
$1$-chromatic by $\phi_{a,b,c}^{\mathrm{{op}}}.$ Hence $\phi_{a,b,c}%
^{\mathrm{{op}}}$ can be extended to a $3$-liec $\phi_{a,b,c}$ of $G.$
Moreover, since every end-grape of $G$ is $A_{1},$ it contains only triangles,
thus the restriction of $\phi_{a,b,c}$ to each such triangle must be a primary coloring.

\medskip\noindent\textbf{Case 2: }$G^{\mathrm{{op}}}$\emph{ is a unicyclic
graph.} Let $C$ be the cycle of $G^{\mathrm{{op}}}.$ Assume first that there
is a triangle in $G$ which hangs at a vertex $u\in V(C),$ denote such triangle
by $G_{0}.$ Let $G^{\mathrm{{lop}}}$ be the graph obtained from $G$ by opening
all triangles of all end-grapes of $G$ except the triangles hanging at $u$ in
$G.$ Notice that $G^{\mathrm{{lop}}}$ is a grape rooted at $u$ with at least
two cycles, so according to Proposition \ref{Prop_grapes2} the grape
$G^{\mathrm{{lop}}}$ admits a $3$-liec such that every berry of
$G^{\mathrm{{lop}}}$ is colored by a primary, secondary or tertiary coloring.
Such a $3$-liec of $G^{\mathrm{{lop}}}$ can be extended to a $3$-liec of $G$
according to Proposition \ref{Prop_ColoringExtendableClosure1}, so the claim
is established.

Assume now that there is no triangle in $G$ hanging at a vertex of $C.$ This
implies every end-grape of $G$ is $A_{1}.$ Let $u^{\prime}$ be the root vertex
of an end-grape $A_{1}$ in $G$ and let $u$ be the vertex from $C$ closest to
$u^{\prime}.$ Let $G^{\mathrm{{lop}}}$ be the graph obtained from $G$ by
opening every end-grape $G_{u^{\prime\prime}}$ of $G$ such that a shortest
path from $u^{\prime\prime}$ to $u^{\prime}$ leads through $u.$ Notice that we
opened at least one end-grape of $G$ to obtain $G^{\mathrm{{lop}}},$ otherwise
$C$ would belong to an end-grape of $G$ and thus opened in $G^{\mathrm{{op}}%
}.$ We conclude that $G^{\mathrm{{lop}}}$ has less cycles than $G.$ Also, the
cycle $C$ is contained in an end-grape $G_{u}$ of $G^{\mathrm{{lop}}}\ $such
that the root component of $G_{u}\ $contains $u^{\prime}$. Since the roots of
all opened end-grapes of $G\ $belong to $G_{u}$ of $G^{\mathrm{{lop}}},$ the
claim is implied by induction hypothesis and Proposition
\ref{Prop_ColoringExtendableClosure1}.

\medskip\noindent\textbf{Case 3: }$G^{\mathrm{{op}}}$\emph{ is a cactus graph
with at least two cycles.} If $G^{\mathrm{{op}}}$ is a grape, the claim is
implied by induction hypothesis and Proposition
\ref{Prop_ColoringExtendableClosure1}. If $G^{\mathrm{{op}}}$ is not a grape,
then it contains an end-grape $G_{u},$ and let $G^{\mathrm{{lop}}}$ be a graph
obtained from $G$ by opening only end-grapes of $G$ whose root vertex belongs
to $G_{u}.$ Then, the claim is again implied by induction hypothesis and
Proposition \ref{Prop_ColoringExtendableClosure1}.
\end{proof}

\end{document}